\documentclass{amsart}
\usepackage{graphicx} 
\usepackage{amsfonts, amsmath, amssymb, amsthm}
\usepackage{ stmaryrd } 
\usepackage[d]{esvect} 
\usepackage{xypic}
\usepackage[margin=1in]{geometry}

\usepackage[dvipsnames]{xcolor} 

\usepackage[hidelinks]{hyperref}
\hypersetup{linktocpage,bookmarksopen = true, bookmarksopenlevel = 1, colorlinks=true,
linkcolor = Bittersweet, citecolor = RoyalBlue , urlcolor = gray, pdfstartview = FitR}

\numberwithin{equation}{section}

\newtheorem{theorem}{Theorem}[section]
\newtheorem{lemma}[theorem]{Lemma}

\newtheorem{corollary}[theorem]{Corollary}
\theoremstyle{definition}
\newtheorem{definition}[theorem]{Definition}
\theoremstyle{remark}
\newtheorem{remark}[theorem]{Remark}

\definecolor{myorange}{rgb}{0.9, 0.55, 0.3}
\definecolor{mygreen}{rgb}{0.35, 0.71, 0.0}
\definecolor{mybrown}{rgb}{0.63, 0.32, 0.18}

\renewcommand{\vec}{\vv}

\newcommand{\QQ}{\mathbb{Q}}
\newcommand{\ZZ}{\mathbb{Z}}

\newcommand{\FF}{\mathbb{F}}

\newcommand{\End}{\operatorname{End}}

\newcommand{\Div}{\operatorname{Div}}
\newcommand{\Ker}{\operatorname{Ker}}

\newcommand{\Hom}{\operatorname{Hom}}
\renewcommand{\div}{\operatorname{div}}


\title{Division polynomials for arbitrary isogenies}

\author[Stange]{Katherine E. Stange}
\address{University of Colorado Boulder, Boulder, Colorado, USA}
\email{kstange@math.colorado.edu}
\urladdr{https://math.katestange.net/}

\begin{document}
\maketitle

\newcommand{\Psis}{\Psi^{\sqrt{\;}}}.

\begin{abstract}
Following work of Mazur-Tate and Satoh, we extend the definition of division polynomials to arbitrary isogenies of elliptic curves, including those whose kernels do not sum to the identity.  In analogy to the classical case of division polynomials for multiplication-by-n, we demonstrate recurrence relations, identities relating to classical elliptic functions, the chain rule describing relationships between division polynomials on source and target curve, and generalizations to higher dimension (i.e., elliptic nets).
\end{abstract}

\section{Introduction}

Given an elliptic curve $E$ with identity $\mathcal{O}$, and a positive integer $n$ with associated multiplication-by-$n$ map $[n]:E \rightarrow E$, the $n$-th division polynomial $\Psi_n$ is an elliptic function on $E$ with divisor
\begin{equation}
  \label{eqn:mazurtate}
  [n]^*(\mathcal{O}) - n^2(\mathcal{O}).
\end{equation}
We typically set $\Psi_0 = 0$ and $\Psi_{-n} = -\Psi_n$.  
Sometimes these are called \emph{Weber polynomials} \cite{Bostan}.
These are furthermore normalized so that they satisfy a recurrence relation
\begin{equation}
	\label{eqn:ellrec}
\Psi_{p + q} \Psi_{p- q}\Psi_{r}^2 \\
+ \Psi_{q+ r}\Psi_{ q- r}\Psi_{ p}^2 \\
+ \Psi_{ r + p} \Psi_{ r- p} \Psi_{ q}^2 = 0,
\end{equation}
or the more general \cite{KateANT}
\begin{equation}
	\label{eqn:ellrec-gen}
\Psi_{p + q + s} \Psi_{p- q}\Psi_{r+ s}\Psi_{r} \\
+ \Psi_{q+ r+ s}\Psi_{ q- r}\Psi_{ p+ s}\Psi_{ p} \\
+ \Psi_{ r + p+ s} \Psi_{ r- p} \Psi_{ q+ s}\Psi_{ q} = 0,
\end{equation}
for all $p,q,r,s \in \ZZ$.  This doesn't completely specify the sequence:  $\Psi_{n}$ can be multiplied by a factor of the form $\alpha^{n^2}\beta$ for scalars $\alpha$ and $\beta$.  Under the normalization we will use throughout, the first few division polynomials, in terms of a Weierstrass curve $y^2 = x^3 + ax + b$, are:
\begin{align*}
	\Psi_1 &= 1,  \quad
	\Psi_2 = -2y, \quad 
	\Psi_3 = 3x^4 + 6a x^2 + 12bx - a^2, \\
	\Psi_4 &= -4y \cdot \left( x^6 +5ax^4 + 20bx^3 - 5a^2x^2 - 4abx - 8b^2 - a^3 \right),
\end{align*}
from which the rest follow by \eqref{eqn:ellrec} or \eqref{eqn:ellrec-gen}.

The recurrence allows for efficient computation, with $O(\log n)$ applications of \eqref{eqn:ellrec-gen} to compute $\Psi_n$ from the initial terms $\Psi_1, \Psi_2, \Psi_3, \Psi_4$.  Ward \cite{Ward} showed in 1948 that integer sequences satisfying \eqref{eqn:ellrec} are essentially those of the form $\Psi_n(P)$ for some point $P$ on an elliptic curve $E$; the curve coefficients and point coordinates can be recovered from the integer sequence as polynomials in the initial terms.  Such integer sequences are known as \emph{elliptic divisibility sequences}.

There are three traditionally important properties of the division polynomials:
\begin{enumerate}
	\item \textbf{Chain rule}.  The $\Psi_n$ satisfy
		\[
			\Psi_{nm} = \left( \Psi_n \circ [m] \right) \Psi_m^{n^2}.
		\]
	\item \textbf{Relation to $x$}.  Letting $x$ be the $x$-coordinate in the Weierstrass form, the $\Psi_n$ satisfy
		\[
			\frac{\Psi_{n+m}\Psi_{n-m}}{\Psi_n^2\Psi_m^2} = x \circ [m] - x \circ [n].
		\]
	\item \textbf{Recurrence relations}.  The $\Psi_n$ satisfy \eqref{eqn:ellrec} as a consequence of the relation to $x$, as well as the more general \eqref{eqn:ellrec-gen}.
\end{enumerate}

Later, Ward studied a generalization to the case of complex multiplication by the Gaussian integers \cite{WardGaussian}.  Just as in his original memoir, he worked in the complex case, defining $\Psi_\alpha := \sigma(\alpha z)/\sigma(z)^{N(\alpha)}$ in terms of the Weierstrass sigma function for $\alpha \in \ZZ[i]$.  These still satisfy the recurrence relation but are not elliptic functions in all cases, so he defines an adjustment which is elliptic but no longer satisfies the recurrence, by scaling by $\sqrt{\wp(z)}$.  Durst studies the Eisenstein case \cite{Durst}.

In 1991, in an appendix to their work on the $p$-adic sigma function, Mazur and Tate defined division polynomials more generally \cite[Appendix I]{MazurTate} as follows.  For any isogeny $\phi: E \rightarrow E'$ for which the divisor
\[
  \phi^*(\mathcal{O}) - (\deg \phi) (\mathcal{O})
\]
is principal, we write $\Psi_\phi$ for a function with this divisor, which we will call the division polynomial for $\phi$.  To set the scalar normalization, let $t$ and $t'$ be uniformizers at the identities $\mathcal{O}$ and $\mathcal{O}'$ for $E$ and $E'$ respectively, and let $\omega$ and $\omega'$ be invariant differentials.  One requires
\[
	\frac{ t^{\deg \phi} \Psi_\phi }{t' \circ \phi} (\mathcal{O}) =  \left( \frac{dt}{\omega} (\mathcal{O}) \right)^{\deg \phi} \left( \frac{dt'}{\omega'} (\mathcal{O}') \right)^{-1}.
\]
This is independent of the choice of $t$ and $t'$ but depends on $\omega, \omega'$.  
As an example, suppose $\phi = [2]$ on $y^2 = x^3 + ax + b$, where $\omega = dx/2y$.  Around $\mathcal{O}$, we can choose $t = t' = -x/y$, with $x = t^{-2} + \cdots$ and $y = -t^{-3} + \cdots$.   
We have $t \circ [2] = 2 t + \cdots$ and $(dt/\omega) (\mathcal{O}) = 1$.  Then, the requirement above becomes $\Psi_2 =  2 t^{-3} + \cdots$.  Thus, combining normalization and divisor, $\Psi_2 = -2y$.  

The requirement of principality, i.e., the sum of the geometric points in the kernel being trivial, is restrictive: the isogenies for which the sum of points is non-zero are exactly those which are cyclic of even degree, with odd inseparable degree.  For example, the endomorphisms $1+i$, $1+\sqrt{-5}$, $\sqrt{-2}$ etc. do not have well-defined division polynomials (see also \cite[Lemma 2.2]{Satoh}).  
For CM by $\ZZ[i]$, these are exactly the cases where Ward involves $\sqrt{\wp(z)}$.

Mazur and Tate show that for $\Psi_\phi$ which are defined, we obtain the three usual properties:  a recurrence relation, the chain rule, and the relation to the $x$-coordinate.

For an isogeny $\phi$, we may also define the closely related \emph{kernel polynomial}
\[
  k_\phi(X) :=  \prod_{\mathcal{O} \neq R \in E[\phi]/\{ \pm 1 \}} (X - x(R)).
\]
For $\phi$ of odd degree, this is again the division polynomial, but for even degree, the multiplicities of two-torsion points may differ; e.g., $\div(\Psi_2) = \div(y) = \sum_{\mathcal{O} \neq R \in E[2]} (R) - 3(\mathcal{O})$ but $\div(k_{[2]}) = 2 \div(\Psi_2)$.  In SageMath, the command \verb|division_polynomial| comes with the option \verb|two_torsion_multiplicity| to choose amongst the variations \cite{SageMath}.

Schoof shows how to compute the kernel polynomial from knowledge of the image curve, the degree, and the sum of the $x$-coordinates of the kernel points (which specifies the second coefficient of the polynomial), by means of Taylor series expansions of Weierstrass functions \cite{Schoof}.  The kernel polynomial's roots can be used to compute the isogeny itself \cite[Equation (8)]{Velu}, and its coefficients determine the target curve \cite[Equation (11)]{Velu}.  The kernel polynomial can be computed by an algorithm of Stark based on continued fraction expansions \cite{Stark}.  For more background on the isogeny computation problem, see \cite{Bostan}.
Evaluations of kernel polynomials can be used to compute the values of isogenies \cite{Faster}.  
In the theory of complex multiplication, ray class fields can be generated over Hilbert class fields by division and kernel polynomials.  
Motivated by this connection, K\"u\c{c}\"uksakalli studied kernel polynomials (calling them generalized division polynomials), and gave a method to compute them using Newton identities and Hurwitz numbers \cite{Omer}.  None of these methods approach the problem using recurrence relations like \eqref{eqn:ellrec}.

In 2004, Satoh independently defined generalized division polynomials for endomorphisms, and studied their computational properties \cite{Satoh}.  Again, one is restricted to the case that the geometric points of the kernel sums to zero; Satoh called such endomorphisms \emph{unbiased}.  In this case, the normalization condition is given in terms of the uniformizer $T = -x/y$ at $\mathcal{O}$ by specifying the leading coefficient of a formal series expansion:
\[
	\Psi_\alpha = (-1)^{N(\alpha)-1}\alpha T^{-N(\alpha)+1} + \cdots.
\]
Again, Satoh recovers the three basic properties of recurrence, chain rule and relation to $x$.  Observe again that Satoh's normalization gives $\Psi_2 = -2y$, agreeing with Mazur-Tate for $y^2 = x^3 + ax + b$ and $\omega = dx/2y$.  As for the history of normalizations, Weber \cite[pp.197-199]{Weber} normalizes so that $\Psi_{2n}$ has lead term $-2nyx^{2n^2-2}$; Cassels \cite{Cassels} chose to change this to a plus sign, which is the most common convention today, e.g. \cite{Sil1}.  Mazur and Tate's convention, equivalently that of Satoh, which we use here, changes the sign back to Weber's convention, at least assuming we take $\omega=dx/(2y+a_1x+a_3)$ for a Weierstrass model \cite[p.682]{MazurTate}.

Over $\QQ$, a very closely related (but not exactly equivalent) definition of an elliptic divisibility sequence is as the sequence of denominators of $[n]P$, $n \ge 0$.  The distinction goes back to Ward \cite{Ward}, but see for example \cite{Verzobio} for modern details on when and how these two definitions differ.  In 2008, Streng \cite{Streng} generalized this alternate definition to curves with complex multiplication, in order to prove a generalization of a property due to Silverman \cite{Sil3} for elliptic divisibility sequences:  that every term has a primitive divisor, that is, a prime divisor not appearing as a divisor earlier in the sequence.  Streng's definition generalizes terms from numbers to ideals.

In 2008, the author generalized elliptic divisibility sequences and division polynomials to elliptic nets \cite{KateANT}.  Net polynomials are polynomials in the coefficients of $E$ and \emph{several} points $P_i$:  the net polynomial $\Psi_{a_1, \ldots, a_n}$ will vanish when $\sum a_i P_i = \mathcal{O}$.

The purpose of this note is to show that the restriction that the sum of points in the kernel be trivial can be circumvented.  We can define division polynomials $\Psi_\phi$ attached to arbitrary isogenies $\phi$, and they satisfy analogues of the three main properties:  chain rule, relation to $x$, and recurrence relation.  Some adjustments to the statements are needed. 

The fundamental idea is to replace the Mazur-Tate divisor with
\[
  D_\phi = \phi^*(\mathcal{O}')
	 {- (\deg\phi) (\mathcal{O})}
+ {(P_\phi)
	- (\mathcal{O})},
\]
where $P_\phi \in E[2]$ is the sum of the kernel of $\phi$.  Then we define an appropriately normalized elliptic function $\Psi_\phi$ having this divisor.  This necessitates a great deal of wrangling of two-torsion points and isogenies of degree $2$.  In particular, the normalization of division polynomials is delicate (the recurrences depend on it), and our case is no exception.  The normalization in the general case is taken with respect to a fixed collection of isogenies of degree $2$.  Aspects of this approach are reminiscent of the theory of theta characteristics and syzygetic triples \cite[Section 1]{Farkas}.

To facilitate all this wrangling, we consider some generalities about \emph{kernel divisors} formed as linear combinations of divisors of the form $\phi^*(\mathcal{O}')$, and, for principal kernel divisors, appropriately normalized \emph{kernel functions}.  We do only what is needed here, but there may be a more general theory available (Section~\ref{sec:kernel}).

One of the interesting waypoints appears in the form of Lemma~\ref{lemma:squareroot}, which turns on the fact that kernel sums play well with the cube law of quadratic forms.  It states that certain `quadratic combinations' of functions supported on two-torsion are squares.

Since the recurrence relations necessitate new factors in our setting, the $\Psi_\phi$ themselves do not form an elliptic divisibility sequence.  However, we demonstrate that specializations to a point $\Psi_\phi(P)$ can recover elliptic divisibility sequences on the target curve (Theorem~\ref{thm:recover}). 

In Section~\ref{sec:examples}, we give some examples.  These demonstrate the three main properties in various contexts, including supersingular and inseparable examples.  In the Gaussian  complex multiplication case, we compare our definitions to Ward's.  

Finally, we generalize such generalized division polynomials -- perhaps `isogeny polynomials' is a better name -- to higher dimension.  That is, we define such things on products of $E$, in analogy to elliptic nets (Section~\ref{sec:higher}).

\subsection*{Acknowledgements.}  Thank you to Joseph Macula and very helpful anonymous referees for feedback on an earlier draft.  The author would like to thank her colleague David Grant for suggesting the footnote, the history of division polynomial normalization, and the alternate proof of Theorem~\ref{thm:relationtox}.

\section{Definitions}

Throughout the paper, we will assume the base field is perfect, algebraically closed, and of characteristic not two.

\subsection{Biased and unbiased isogenies}
Generalizing Satoh, we call a non-zero isogeny $\phi: E \rightarrow E'$ \emph{unbiased} if its kernel elements sum to the identity with multiplicity \cite[Section 2]{Satoh}.  In other words, the map $\sum_i n_i (P_i) \mapsto \sum_i n_i P_i$ taking $\Div(E)$ to $E$ takes $\phi^*(\mathcal{O}')$ to $\mathcal{O}$.  Equivalently, the divisor
\[
  \phi^*(\mathcal{O}') - (\deg \phi) (\mathcal{O})
\]
is principal.  By convention, the zero isogeny, which we will write as $[0]$ to avoid confusion, is unbiased.

A non-zero isogeny is \emph{biased} if its kernel sum is non-trivial; in this case, it must sum to a non-trivial two-torsion point (all points of higher order are distinct from their inverses and will cancel one another).  We will denote this point $P_\phi$.
For a biased isogeny $\phi: E \rightarrow E'$, the divisor
\[
  \phi^*(\mathcal{O}') - (\deg \phi) (\mathcal{O})
\]
is no longer principal.  There are three `nearby' principal divisors which may prove useful:
\[
  2 \left( \phi^*(\mathcal{O}') - (\deg \phi) (\mathcal{O}) \right), \quad
    [2]^* \left( \phi^*(\mathcal{O}') - (\deg \phi) (\mathcal{O}) \right), \quad \text{and} \quad
  D_\phi := \phi^*(\mathcal{O}')
	 {- (\deg\phi) (\mathcal{O})}
+ {(P_\phi)
	- (\mathcal{O})}.
    \]
    This last, denoted $D_\phi$, may for applications be the `best' replacement for $ 
    \phi^*(\mathcal{O}') - (\deg \phi) (\mathcal{O})$, in the sense that it is the minimal modification.  But the first two will also have a role to play\footnote{These three modifications correspond respectively, in the arithmetic setting of a non-principal ideal in a number field, to taking a principal power; extending to a field where the ideal becomes principal; and multiplying by a fixed representative of the inverse ideal class.}:  see Remark~\ref{rem:tilde} and the alternate proof to Theorem~\ref{thm:relationtox}.

\subsection{Three isogenies of degree two}
Write $P_0 = \mathcal{O}, P_1, P_2, P_3$ for the points in $E[2]$.  For each $i=1,2,3$, let $g_i : E \rightarrow E_i$ denote a fixed isogeny of degree $2$ with kernel $\{ \mathcal{O}, P_i \}$.  Let $g_0$ denote the identity isogeny on $E_0 := E$.  Let $t$ and $t_i$ be uniformizers at the identities on $E$ and $E_i$ respectively.  Let $\omega$ and $\omega_i$ be invariant differentials on $E$ and $E_i$ respectively.  And finally, write $\mathcal{O}$ for the identity on $E$ and $\mathcal{O}_i$ for the identity on $E_i$.  By convention, choose $t_0=t$, and $\omega_0=\omega$.

Let $\phi: E \rightarrow E'$ be a non-zero isogeny.  Let $P_\phi \in E[2]$ be the sum of the points in the kernel of $\phi$.
Write $\iota(\phi)$ for the $i \in \{0,1,2,3\}$ such that $P_\phi = P_i$.
Then we have or define
\[
	P_\phi = P_{\iota(\phi)},\quad g_\phi := g_{\iota(\phi)}, \quad  t_\phi := t_{\iota(\phi)},\quad \omega_\phi := \omega_{\iota(\phi)},\quad \mathcal{O}_\phi := \mathcal{O}_{\iota(\phi)}.
\]
Observe that $\iota(\phi)$ depends only on $\phi|_{E[2]}$,  
since $P_\phi = \sum_{P \in E[\phi]} P = \sum_{P \in E[\phi] \cap E[2]} P$.  By convention, we also set 
\[
E_{[0]} := E, \quad P_{[0]} := \mathcal{O}, \quad g_{[0]} := [1], \quad t_{[0]} := t, \quad \omega_{[0]} := \omega.
\]

\subsection{The elliptic function $\Psi_\phi$}
\label{sec:psi}
We now extend the definition of division polynomials to arbitrary (possibly biased) isogenies.
Let $\phi: E \rightarrow E'$ be non-zero, let $t'$ be a uniformizer at the identity and $\omega'$ an invariant differential for $E'$.  Let $\mathcal{O}'$ be the identity of $E'$.
In analogy to Mazur and Tate, define an elliptic function $\Psi_\phi$ with the principal divisor
\begin{equation}
  \label{eqn:Dphi}
  D_\phi = \phi^*(\mathcal{O}')
	 {- (\deg\phi) (\mathcal{O})}
+ {(P_\phi)
	- (\mathcal{O})},
\end{equation}
where we require, as a means of normalization, 
that
\begin{equation}
	\label{eqn:norm-arb}
	\frac{ t^{\deg \phi+\deg g_\phi} \Psi_\phi }{(t_\phi \circ g_\phi)(t' \circ \phi)} (\mathcal{O}) 
	=  
	\left( \frac{dt}{\omega} (\mathcal{O}) \right)^{\deg \phi+\deg g_\phi} 
	\left( \frac{dt'}{\omega'} (\mathcal{O}') \right)^{-1}
	\left( \frac{dt_\phi}{\omega_\phi} (\mathcal{O}_\phi) \right)^{-1}.
\end{equation}
One observes that this is independent of $t, t', t_\phi$ but depends on $\omega, \omega', \omega_\phi$.  As an example, for $E: y^2 = x^3 + x$ with complex multiplication by $\ZZ[i]$, $\Psi_{1+i} = 2ix$ has divisor $2(0,0) - 2(\mathcal{O})$ (examples are given in Section~\ref{sec:examples}).

By (somewhat paradoxical) convention, we set $\Psi_{[0]} = \Psi_0 = 0$ and $D_{[0]} = \emptyset$.

\subsection{The elliptic function $\widehat{\Psi}_\phi$}
We need some auxiliary functions that will `correct' the recurrence equations for biased isogenies.
Define $\widehat{\Psi}_i := \Psi_{g_i}$ for $i=1,2,3$ and $\widehat{\Psi}_0 := 1$.  
Then for a general isogeny, define
\[
 \widehat{\Psi}_\phi := \widehat{\Psi}_{\iota(\phi)}.
\]
To be explicit, for each $i=0, 1, 2, 3$, we equivalently define an elliptic function $\widehat{\Psi}_i$ on $E$ with the principal divisor
\[
  \div(\widehat{\Psi}_i) = \widehat{D}_i := 2(P_i) - 2(\mathcal{O}), 
\]
normalized so that
\begin{equation*}
  \frac{ t^{2\deg g_i} \widehat{\Psi}_i}{(t_i \circ g_i)^2} (\mathcal{O}) 
	=  
  \left( \frac{dt}{\omega} (\mathcal{O}) \right)^{2 \deg g_i}
	\left( \frac{dt_i}{\omega_i} (\mathcal{O}_i) \right)^{-2}.
\end{equation*}
As usual, this is independent of $t, t_i$ but depends on the invariant differentials.

\subsection{Formul\ae}
With the definitions as in the last section, we will prove the following formul\ae:
\begin{enumerate}
  \item First chain rule for unbiased $\alpha : E \rightarrow E'$, $\beta : E'' \rightarrow E$ (Theorem~\ref{thm:chain}):
	\[
			\Psi_{\alpha\beta}
			=
		\left(\Psi_\alpha \circ \beta\right) \Psi_\beta^{\deg \alpha}.
	\]
\item First chain rule for possibly biased $\alpha$, $\beta$ (Theorem ~\ref{thm:chain}):
  \[
    \left(
		\frac{
			\Psi_{\alpha\beta}
		}{
		\left(\Psi_\alpha \circ \beta\right) \Psi_\beta^{\deg \alpha}
	}
  \right)^2
=
{\frac{
				\widehat\Psi_{\alpha\beta}
			}{
				(\widehat\Psi_\alpha \circ \beta)
				\widehat\Psi_\beta^{\deg \alpha}
    }}.
	\]
  See also the Second Chain Rule (Theorem~\ref{thm:secondchain}).
  \item First relation to $x$ (Theorem~\ref{thm:relationtox}): For non-zero $\alpha, \beta: E \rightarrow E'$,
    \[
		 \frac{\Psi_{\alpha + \beta} \Psi_{\alpha - \beta} \widehat{\Psi}_\alpha \widehat{\Psi}_\beta}{ \Psi_\alpha^2 \Psi_\beta^2  \widehat{\Psi}_{\alpha + \beta} } = x' \circ \beta - x' \circ \alpha.
   \]
  \item First recurrence relation (Corollary \ref{cor:rec}): For $\alpha, \beta, \gamma : E \rightarrow E'$ for which the following are defined,
    \[
		 \frac{\Psi_{\alpha + \beta} \Psi_{\alpha - \beta} \widehat{\Psi}_\alpha \widehat{\Psi}_\beta}{ \Psi_\alpha^2 \Psi_\beta^2  \widehat{\Psi}_{\alpha + \beta} } 
		 +\frac{\Psi_{\beta + \gamma} \Psi_{\beta - \gamma} \widehat{\Psi}_\beta \widehat{\Psi}_\gamma}{ \Psi_\beta^2 \Psi_\gamma^2  \widehat{\Psi}_{\beta + \gamma} } 
		 +\frac{\Psi_{\gamma + \alpha} \Psi_{\gamma - \alpha} \widehat{\Psi}_\gamma \widehat{\Psi}_\alpha}{ \Psi_\gamma^2 \Psi_\alpha^2  \widehat{\Psi}_{\gamma + \alpha} } 
		 =0.
   \]
  \item Second relation to $x$ (Theorem \ref{thm:relationtox2}): For $\alpha, \beta, \sigma: E \rightarrow E'$ for which the following are defined,
    \[
      \frac{
      \Psi_{\alpha + \beta + \sigma} \Psi_{\alpha - \beta} \Psi_{\sigma}
	 }{
		 {\Psi}_{\alpha+\sigma} {\Psi}_{\beta+\sigma} \Psi_\alpha \Psi_\beta 
	 } 
	 \sqrt{\frac{
		 \widehat{\Psi}_{\alpha+\sigma} \widehat{\Psi}_{\beta+\sigma} \widehat\Psi_\alpha \widehat\Psi_\beta 
	 }{
     \widehat\Psi_{\alpha + \beta + \sigma} \widehat\Psi_{\alpha - \beta}\widehat{\Psi}_\sigma
	 }}
   = \frac{
     y' \circ \beta - y' \circ \sigma
   }{
     x' \circ \beta - x' \circ \sigma
   }
 -
\frac{
     y' \circ \alpha - y' \circ \sigma
   }{
     x' \circ \alpha - x' \circ \sigma
   }
   .
 \]
  \item Second recurrence relation (Corollary \ref{cor:rec2}): For $\alpha, \beta, \gamma, \sigma: E \rightarrow E'$ for which the following are defined,
    \begin{equation*}
      \begin{split}
  \frac{
		 \Psi_{\alpha + \beta + \sigma} \Psi_{\alpha - \beta}		
	 }{
		 {\Psi}_{\alpha+\sigma} {\Psi}_{\beta+\sigma} \Psi_\alpha \Psi_\beta 
	 } 
	 \sqrt{\frac{
		 \widehat{\Psi}_{\alpha+\sigma} \widehat{\Psi}_{\beta+\sigma} \widehat\Psi_\alpha \widehat\Psi_\beta 
	 }{
		 \widehat\Psi_{\alpha + \beta + \sigma} \widehat\Psi_{\alpha - \beta}		
	 }}
   &+
	\frac{
		 \Psi_{\beta + \gamma + \sigma} \Psi_{\beta - \gamma}	
	 }{
		 {\Psi}_{\beta+\sigma} {\Psi}_{\gamma+\sigma} \Psi_\beta \Psi_\gamma 
	 } 
	 \sqrt{\frac{
		 \widehat{\Psi}_{\beta+\sigma} \widehat{\Psi}_{\gamma+\sigma} \widehat\Psi_\beta \widehat\Psi_\gamma 
	 }{
		 \widehat\Psi_{\beta + \gamma + \sigma} \widehat\Psi_{\beta - \gamma}	
	 }} \\
   &+
	\frac{
		\Psi_{\gamma + \alpha + \sigma} \Psi_{\gamma - \alpha}
	 }{
		 {\Psi}_{\gamma+\sigma} {\Psi}_{\alpha+\sigma} \Psi_\gamma \Psi_\alpha 
	 }
	 \sqrt{\frac{
		 \widehat{\Psi}_{\gamma+\sigma} \widehat{\Psi}_{\alpha+\sigma} \widehat\Psi_\gamma \widehat\Psi_\alpha 
	 }{
		 \widehat\Psi_{\gamma + \alpha + \sigma} \widehat\Psi_{\gamma - \alpha}	
	 }}
=0.
\end{split}
\end{equation*}
  \end{enumerate}
  The role and well-definition of the extra square-root factors are explained by Lemmas~\ref{lemma:squareroot} and \ref{lemma:quadrel}.

\section{Kernel functions and the chain rule}
\label{sec:kernel}

The Mazur-Tate method for normalizing division polynomials can be applied generally to a class of functions which we will call \emph{kernel functions}, and gives rise to a system of functions which are well-normalized with respect to one another.  Using this language, we obtain the chain rule and its generalization.

Given an isogeny $\phi: E \rightarrow E'$, define the divisor
\[
  K_\phi := \phi^*(\mathcal{O}') \in \Div(E).
\]
Conventionally, $K_{[0]} = \emptyset$ for any zero isogeny $[0]$.
Distinct from the divisor $K_\phi$, we will write $(K_\phi)$ as a formal \emph{kernel symbol} associated to $\phi$.
We can define the group of formal $\ZZ$-sums of kernel symbols $(K_\phi)$,
\[
  \Ker(E) := \left\{ \sum_\phi n_\phi (K_\phi) : \substack{\phi \text{ ranges over non-zero isogenies with source $E$} \\ \text{and all but finitely many $n_\phi \in \ZZ$ are zero.} } \right\}.
\]
Its elements will be called \emph{kernel symbol sums}.  Conventionally, for any zero isogeny $[0]$, the symbol $(K_{[0]})$ should be interpreted as the zero element of $\Ker(E)$ (the empty sum).  There is a map from $\Ker(E)$ to $\Div(E)$ by substituting the kernel divisor $K_\phi$ for each symbol $(K_\phi)$, namely 
\[
  \sum_\phi n_\phi (K_\phi) \mapsto \sum_\phi n_\phi K_\phi.
\]
Throughout the paper, a sum or product with a prime, such as $\sum'_\phi$, indicates that zero isogenies are dropped from the indexing set.

Given a kernel symbol sum whose image in $\Div(E)$ is principal, we can associate an elliptic function with that image, up to scalar normalization.  The choice of normalization is more delicate, depending on the kernel symbol sum and not just the image divisor.  
Fix a uniformizer $t_\phi$ at the identity and differential $\omega_\phi$ on the target curve for each $\phi$.  For now, for the greatest generality, we do this independently for each $\phi$, even if some target curves coincide.
To a principal kernel symbol sum $\sum_\phi n_\phi (K_\phi)$, the associated elliptic function $h$ can be normalized by requiring
\begin{equation}
\label{eqn:normalizeme}
	\frac{h}{
	\left(\prod'_\phi (t_\phi \circ \phi)^{n_\phi}\right)
}
(\mathcal{O}) = \prod_\phi{\vphantom{\sum}}' \left( \frac{dt_\phi}{\omega_\phi}\left(\mathcal{O}_\phi\right)\right)^{-n_\phi},
\end{equation}
where by design both sides have a finite non-zero value.  This is independent of the choice of uniformizers but depends on the choice of invariant differentials. 
Any elliptic function $h$ of this form and normalized in this way is called the \emph{kernel function} for the associated kernel symbol sum.  

This normalization is consistent in the sense that the product of two kernel functions is again a kernel function (for the sum of the kernel symbol sums).  

It has the following convenient property.
\begin{lemma}
  \label{lemma:independent}
	Fix an elliptic curve $E$.  
  For each other elliptic curve $E'$, suppose $\omega' := \omega_\phi$ are chosen to agree for all $\phi \in \Hom(E,E')$.  
  Consider 
  a kernel symbol sum $\sum_{\phi \in I} n_\phi (K_\phi)$ indexed over some collection $I$ of isogenies out of $E$.
  Then the derived kernel function is independent of the choice of invariant differential $\omega'$ on $E'$ whenever $\sum_{\phi\in I \cap \Hom(E,E')} n_\phi = 0$.  
\end{lemma}

Returning to $\Psi_\phi$, for non-zero $\phi$, observe that we have defined it in \eqref{eqn:Dphi} by kernel symbol sum
\begin{equation}
  \label{eqn:DphiK}
	(K_\phi) + (K_{g_\phi}) - (\deg \phi + \deg g_\phi) (K_1)
\end{equation}
(giving rise to divisor $D_\phi$), and the associated normalization; it is therefore a kernel function with respect to the given kernel symbol sum.
The auxiliary functions $\widehat{\Psi}_\phi$ (defined above) and $\widetilde{\Psi}_\phi$ (discussed in Remark~\ref{rem:tilde}) are also kernel functions, with divisors
\[
  \div(\widehat{\Psi}_\phi) = 2(P_\phi) - 2(\mathcal{O}) = 2(K_{g_\phi} - (\deg g_\phi) K_1), \quad
  \div(\widetilde{\Psi}_\phi) = 2 \phi^*(\mathcal{O}') {- 2 (\deg\phi) (\mathcal{O})} = 2 ( K_\phi - (\deg\phi) K_1).
\]
and normalizations controlled by the kernel symbol sums 
\[
  2(K_{g_\phi}) - 2(\deg g_\phi) (K_1), \quad \text{and} \quad
  2 ( K_\phi) - 2(\deg \phi) (K_1),
\]
respectively.

Because of the importance of the isogenies of degree two, we will make a further convention on the choice of invariant differentials, i.e., we choose $\omega_1$, $\omega_2$, $\omega_3$ and $\omega$ so that the kernel function associated to the kernel symbol sum
	\begin{equation}
		\label{eqn:kersum2}
		(K_{g_1}) + (K_{g_2}) + (K_{g_3}) - (K_{[2]}) - 2(K_1)
	\end{equation}
  is $1$.  Namely,
  \begin{equation}
    \label{eqn:convention}
    \frac{(t \circ [2])t^2}{(t_1 \circ g_1)(t_2 \circ g_2)(t_3 \circ g_3)}\left(\mathcal{O}\right) = 
    \left( \frac{dt_1}{\omega_1} \left( \mathcal{O} \right) \right)^{-1}
    \left( \frac{dt_2}{\omega_2} \left( \mathcal{O} \right) \right)^{-1}
    \left( \frac{dt_3}{\omega_3} \left( \mathcal{O} \right) \right)^{-1}
    \left( \frac{dt}{\omega} \left( \mathcal{O} \right) \right)^{3}.
  \end{equation}
  In other words, we accomplish this by scaling the $\omega_i$, $\omega$ relative to one another.
  This convention allows us to consider kernel symbol sums equivalent modulo the expression \eqref{eqn:kersum2}
	when normalizing kernel functions.  

  A word of caution:  $(K_{\phi})$ and $(K_{-\phi})$ are examples of distinct kernel symbols whose divisors $K_\phi$ and $K_{-\phi}$ are equal, but which may induce a different normalization when included in a kernel symbol sum used to generate a kernel function.  In particular, $t \circ (-\phi) = - t\circ \phi$, altering the normalization \eqref{eqn:normalizeme} by a sign.  The same applies to any post-composition by an automorphism.

  It would be interesting to describe the kernel of $\Ker(E) \rightarrow \Div(E)$ more generally; for related literature, see \cite{RaoReid}.

There is a natural notion of pullback on kernel symbol sums, namely, when $\beta: E'' \rightarrow E$,
\[
  \beta^* (K_\gamma) := (K_{\gamma\beta}),
\]
extended linearly.  This commutes with pullback on divisors.

\begin{lemma}
  \label{lemma:pullback}
Let $E$ be an elliptic curve. 
  For each other elliptic curve $E'$, suppose $\omega' := \omega_\phi$ are chosen to agree for all $\phi \in \Hom(E,E')$.  
  Suppose $f$ and $g$ are kernel functions associated to kernel symbol sums $\sum_{\gamma \in \Hom(E,E')} n_\gamma (K_\gamma)$ and $\sum_{\delta \in \Hom(E'',E')} m_\delta (K_\delta)$, respectively.  Suppose $\beta : E'' \rightarrow E$.
  Suppose that $\sum_\delta m_\delta (K_\delta) = \beta^* \sum_\gamma n_\gamma (K_\gamma)$.
  Then $f \circ \beta = g$.  
\end{lemma}

  That is, the pullback of a kernel function is a kernel function for the pullback of its kernel symbol sum.

\begin{proof}
  The assumptions imply $\beta^* \div f = \div g$.
 	Now we consider the normalization.  Since $f$ and $g$ are kernel functions,
  \[
    \frac{g}{ \prod_\gamma (t' \circ \gamma \beta)^{n_\gamma}}(\mathcal{O}'')
    = \left(
      \frac{ dt'}{\omega'}(\mathcal{O}')
    \right)^{-\sum_\gamma n_\gamma},
  \]
  and
\[
  \frac{f \circ \beta}{ \prod_\gamma (t' \circ \gamma \circ \beta)^{n_\gamma}}(\mathcal{O}'')
    = \frac{f}{ \prod_\gamma (t' \circ \gamma)^{n_\gamma}}(\mathcal{O})
    = \left(
      \frac{ dt'}{\omega'}(\mathcal{O}')
    \right)^{-\sum_\gamma n_\gamma}.
  \]
Therefore, the normalizations agree.
\end{proof}

\begin{theorem}[First chain rule]
  \label{thm:chain}
	If non-zero $\alpha: E' \rightarrow E''$ and $\beta: E \rightarrow E'$ are unbiased, then
	\[
			\Psi_{\alpha\beta}
			=
		\left(\Psi_\alpha \circ \beta\right) \Psi_\beta^{\deg \alpha}.
	\]
	In general, we have
	\[
    \left(
		\frac{
			\Psi_{\alpha\beta}
		}{
		\left(\Psi_\alpha \circ \beta\right) \Psi_\beta^{\deg \alpha}
	}
  \right)^2
=
{\frac{
				\widehat\Psi_{\alpha\beta}
			}{
				(\widehat\Psi_\alpha \circ \beta)
				\widehat\Psi_\beta^{\deg \alpha}
    }}.
	\]
\end{theorem}

\begin{proof}
  We begin with the second equation.  Using \eqref{eqn:DphiK}, the left side is a kernel function with kernel symbol sum
  \begin{align*}
    &2\large(
    (K_{\alpha\beta}) + (K_{g_{\alpha\beta}}) - \beta^* (K_\alpha) - \beta^* (K_{g_\alpha}) - \deg \alpha (K_\beta) - \deg \alpha (K_{g_\beta}) \\
    &- (\deg \alpha\beta + \deg g_{\alpha\beta}) (K_1)
    + \beta^* (\deg \alpha + \deg g_{\alpha}) (K_1) + (\deg \alpha \deg \beta + \deg \alpha \deg g_\beta )(K_1) \large)\\
    &= 
    2\large( (K_{g_{\alpha\beta}}) - \beta^* (K_{g_\alpha}) - \deg \alpha (K_{g_\beta}) + \beta^* \deg g_{\alpha} (K_1) - (\deg g_{\alpha\beta} - \deg \alpha \deg g_\beta )(K_1) \large) \\
    &+ 2(K_{\alpha\beta})  - 2\deg \alpha(K_\beta) - 2\beta^*(K_\alpha)- 2\beta^*\deg \alpha (K_1).
  \end{align*}
  We have used that $\deg(\alpha \beta) = \deg \alpha \deg \beta$.  
  Using the definition of pullback on kernel symbol sums, this simplifies to
  \begin{equation*}
    2\large( (K_{g_{\alpha\beta}}) - \beta^* (K_{g_\alpha}) - \deg \alpha (K_{g_\beta}) + \beta^* \deg g_{\alpha} (K_1) - (\deg g_{\alpha\beta} - \deg \alpha \deg g_\beta )(K_1) \large).
  \end{equation*}
  The right side is a kernel function whose kernel symbol sum is this same quantity.
  Therefore the two sides have the same divisor and normalization.
  The first equation is calculated similarly (but more simply).
\end{proof}

\section{Functions supported on two-torsion}
\label{sec:two-tor}

The divisor of $\Psi_\phi$ differs from the `ideal' but non-principal divisor $\phi^*(\mathcal{O}') - (\deg \phi) (\mathcal{O})$ by half of $\div(\widehat{\Psi}_\phi)$, which is again non-principal.  In this section we prove some generalities on when the `extra two torsion' summands combine to give a principal divisor.

  When we write $\sum_\phi$ we will take this as an abbreviation for $\sum_{\phi \in \Hom(E,E')}$ for a fixed $E$ and $E'$.  We also define a \emph{finite integral quadratic identity} to be any identity of the form $\sum_{i\in I} a_i q(i) = 0$, having finitely many terms, indexed over a $\ZZ$-module $I$ with coefficients $a_i \in \ZZ$, which is a combination of finitely many instances of the cube identity:
	\[
		q(\alpha + \beta  + \gamma) - q(\alpha + \beta)
		- q(\beta + \gamma)
		- q(\gamma + \alpha)
		+ q(\alpha) + q(\beta) + q(\gamma) - q(0)
		=0,
	\]
One example is the better-known parallelogram rule,
\[
  q(\alpha + \beta) + q(\alpha - \beta) - 2q(\alpha) - 2q(\beta) + 2q(0) = 0.
\]
Such identities are true for any function which is a linear combination of a quadratic form and a constant.

Recall that a prime on sums and products indicates that zero isogenies are omitted from the indexing set.

\begin{lemma}
	\label{lemma:squareroot}
  Fix $E$ and $E'$.
	Suppose that
	\[
		\sum_\phi e_\phi q(\phi) = 0
		\]
		is a finite integral quadratic identity on $\phi \in \Hom(E,E')$.  Then there is a kernel function $f$ given by 
    \[
      \sum_\phi{\vphantom{\sum}}' e_\phi (K_{g_\phi}) 
    \]
    which has as its square
		\[
			 \prod_\phi{\vphantom{\sum}}' \widehat{\Psi}_\phi^{e_\phi}.
		\]
\end{lemma}

		We will denote $f$ by
		\[
		 \sqrt{ \prod_\phi{\vphantom{\sum}}' \widehat{\Psi}_\phi^{e_\phi}}.
		\]

\begin{proof}
	Observe that the function 
		$\prod'_\phi \widehat{\Psi}_\phi^{e_\phi}$ is the kernel function of the kernel symbol sum
	\[
		2 \left(\sum_\phi{\vphantom{\sum}}' e_\phi (K_{g_\phi}) \right) - 2\left(\sum_\phi{\vphantom{\sum}}' e_\phi (\deg g_\phi) \right) (K_1).
	\]
  It will suffice to show that $\sum'_\phi e_\phi K_{g_\phi}$ is a principal divisor.  
  This ensures that $f$ exists, and since the only multiple of $K_1 = (\mathcal{O})$ which is principal is the trivial multiple, it also shows that
		$\prod'_\phi \widehat{\Psi}_\phi^{e_\phi}$ is the kernel function of the kernel symbol sum
	\[
		2 \left(\sum_\phi{\vphantom{\sum}}' e_\phi (K_{g_\phi}) \right),
	\]
  which implies the statement about the square of $f$.

  Therefore we show that $\sum'_\phi e_\phi K_{g_\phi}$ is a principal divisor.
We may assume without loss of generality that we are in the case of the cube identity
	\[
    \sum_\phi e_\phi q(\phi) =
		q(\alpha + \beta  + \gamma) - q(\alpha + \beta)
		- q(\beta + \gamma)
		- q(\gamma + \alpha)
		+ q(\alpha) + q(\beta) + q(\gamma) - q(0)
		=0,
  \]
  where $\alpha, \beta, \gamma \in \Hom(E,E')$.  The degree of the corresponding divisor is zero, and it is supported on two-torsion points.
  Write $\sum'_\phi e_\phi K_{g_\phi} =: \sum_{i=0}^3 a_i (P_i)$.  
  We will compute the parity of the $a_i$, $i=1,2,3$; it will suffice to observe that the $a_i$ all have the same parity.
 The parity of $a_i$ is the parity of the number of elements $\phi$ in the list
 \begin{equation}
   \label{eqn:list}
    0, \alpha, \beta, \gamma, \alpha+\beta, \beta + \gamma, \gamma + \alpha, \alpha + \beta + \gamma
  \end{equation}
  for which $P_\phi = P_i$.

  To determine $a_i$, we examine the action of each $\phi = \alpha, \beta, \gamma$ on the $2$-torsion, as a map $E[2] \rightarrow E'[2]$ on $\FF_2$-vector spaces.  Denote this by $\rho(\phi) \in \Hom_{\FF_2}(E[2],E'[2])$.  The map $\kappa_i: \Hom_{\FF_2}(E[2],E'[2]) \rightarrow E'[2]$ taking $F$ to $F(P_i)$ is a linear transformation of vector spaces over $\FF_2$ from dimension $4$ to dimension $2$.  We have $P_\phi = P_i$ if and only if $\rho(\phi) \in \ker \kappa_i \setminus \{ 0 \}$.

  Let $C$ be the subspace of $\Hom_{\FF_2}(E[2],E'[2])$ generated by $\rho(\alpha)$, $\rho(\beta)$ and $\rho(\gamma)$, namely
  \[
    C = \{ 0, \rho(\alpha), \rho(\beta), \rho(\gamma), \rho(\alpha + \beta), \rho(\beta + \gamma), \rho(\gamma + \alpha), \rho(\alpha + \gamma + \beta) \}.
  \]
If it is of dimension $3$ (so $C$ has 8 distinct elements), then we conclude from linearity that the fibres of $\kappa_i |_C$ are of even size, and so all the $a_i$ are odd.  

On the other hand, if $C$ is of lower dimension, then each element in the image under $\rho$ of the list \eqref{eqn:list} occurs with even multiplicity.  Therefore $P_\phi = P_i$ occurs an even number of times amongst $\phi$ in \eqref{eqn:list}.  

We have shown that the parity of $a_1, a_2, a_3$ agree.
\end{proof}

\begin{lemma}
	\label{lemma:quadrel}
	If
	\[
		\sum_\phi e_\phi q(\phi) = 0
		\]
		is a finite integral quadratic identity, then
		\[
		\prod_\phi{\vphantom{\sum}}' \Psi_\phi^{e_\phi} \sqrt{ \prod_\phi{\vphantom{\sum}}' \widehat{\Psi}_\phi^{-e_\phi}}
		\]
		is a well-defined kernel function, with divisor
		\[
			\sum_\phi{\vphantom{\sum}}' e_\phi K_\phi.
    \]
\end{lemma}

\begin{proof}
  That it is well-defined follows from Lemma~\ref{lemma:squareroot}; that it has the given divisor is a calculation using the fact that $\sqrt{ \prod'_\phi \widehat{\Psi}_\phi^{e_\phi}}$ has divisor $\sum'_\phi e_\phi K_{g_\phi}$ (from the proof of Lemma~\ref{lemma:squareroot}).
\end{proof}

The First Chain Rule (Theorem ~\ref{thm:chain}) in the biased case is somewhat unsatisfying; however, in certain cases the relationship simplifies. The next result follows immediately from Lemma~\ref{lemma:quadrel}, Lemma~\ref{lemma:pullback} and Lemma~\ref{lemma:independent}.

\begin{theorem}[Second chain rule]
	\label{thm:secondchain}
	Suppose 
 $\sum_{\alpha \in \Hom(E'',E')} e_{\alpha} q(\alpha) = 0$ is a finite integral quadratic identity.
	Let $\beta: E'' \rightarrow E$ be non-zero, such that $e_\alpha = 0$ for $\alpha$ not factoring through $\beta$.
Then, letting $\gamma$ range through $\Hom(E,E')$,
	\[
		\left(\prod_{\gamma}{\vphantom{\sum}}' \Psi^{e_{\gamma\beta}}_{\gamma}
		\sqrt{\prod_{\gamma}{\vphantom{\sum}}' \widehat{\Psi}^{e_{\gamma\beta}}_{\gamma}}\right) \circ \beta
		=
		\prod_{\gamma}{\vphantom{\sum}}' \Psi^{e_{\gamma\beta}}_{\gamma\beta}
		\sqrt{\prod_{\gamma}{\vphantom{\sum}}' \widehat{\Psi}^{e_{\gamma\beta}}_{\gamma\beta}}.
	\]
	The quantity is independent of the choice of invariant differential on $E$ and $E''$.
  If, in addition, $\sum'_{\gamma} e_{\gamma\beta} = 0$, then the quantity in question is independent of the choice of invariant differential on $E'$.
\end{theorem}

\section{Recurrences and Relations to $x$}

\label{sec:recmain}

\subsection{Normalization in terms of formal groups}
\label{sec:formal}

If we choose the normalizer $T := -x/y$ for the Weierstrass form $E: y^2 = f(x)$ over a field, and similarly $T'$ for $E'$ and $T_i$ for $E_i$, then we can specify the normalizations of $\Psi_\phi$ etc. in terms of the formal groups.  We have expansions for the normalized invariant differential, formal group law and $x$- and $y$-coordinates:
\begin{equation}
  \label{eqn:normalized-w}
	\omega = \frac{dT}{1 +O(T)}, 
	\quad
  F(T_1,T_2) = T_1 + T_2 + \cdots, \quad x(T) = T^{-2} + O(T^{-1}), \quad y(T) = -T^{-3} + O(T^{-2}).  
\end{equation}
We also have that there are non-zero constants $a_\phi$ and $a_i := a_{g_i}$ so that
	\begin{equation}
    \label{eqn:leading}
    \phi T = a_\phi T^{\deg_{in} \phi} + O(T^{2\deg_{in} \phi }), \quad
    g_i T = a_i T + O(T^{2 }),
	\end{equation}
  where $\deg_{in}$ represents the inseparable degree (which does not figure in the case for $g_i$ because we are away from characteristic two).
So the normalization \eqref{eqn:norm-arb}, taking the normalized differential as above, becomes
\begin{align}
	\label{eqn:norm}
  \Psi_\phi(T) &= T^{-\deg \phi-2} 
  (\phi T)
  (g_{\iota(\phi)} T) \notag \\
               &= a_{\iota(\phi)} a_\phi T^{-\deg \phi+\deg_{in}\phi-1} + \cdots .
\end{align}
Similarly,
\begin{align}
	\label{eqn:normhat}
  \widehat\Psi_\phi(T) &= T^{-2}
  (g_{\iota(\phi)} T)^2 \notag \\
                       &= a_{\iota(\phi)}^2 T^{-2} + \cdots .
\end{align}

Satoh \cite[Section 3]{Satoh} uses the same divisor to define $\Psi_\alpha$ for unbiased inseparable $\alpha \in \End(E)$, but specifies the normalization by requiring
\[
	\Psi_\alpha = (-1)^{N(\alpha)-1}\alpha T^{-N(\alpha)+1} + \cdots.
\]
Recall that for such $\alpha$, one has $a_\alpha = \alpha$, so that the normalizations of Mazur-Tate and Satoh differ by a sign of the form $(-1)^{\deg \phi - 1}$.  As briefly discussed in the introduction, including this sign does not affect \eqref{eqn:ellrec} or \eqref{eqn:ellrec-gen}, nor does it affect the relation to $x$ or the chain rule.  It is just a convention.

\subsection{Recurrences and Relations}
\label{sec:rec}

For everything that follows, we will now fix five invariant differentials:  $\omega$ on $E$, $\omega'$ on $E'$, $\omega_i$ on $E_i$ (regardless of whether some of these curves coincide).  We will simultaneously fix Weierstrass equations for $E, E', E_i$ for which these are the normalized invariant differentials as in Section~\ref{sec:formal}.  
These are used for normalization consistently as described for $\Psi$, $\widehat{\Psi}$, subject to the constraint \eqref{eqn:convention}.
Finally, we will use the expressions $x', y'$ for the coordinate functions associated to the Weierstrass equation for $E'$.  

The following relation generalizes the usual relation to $x$ in the unbiased case (described in the Introduction), where the $\widehat{\Psi}$-terms become trivial.  One should think of this formula as demonstrating the philosophy that, while for biased $\alpha$, the $\Psi_\alpha$ we want isn't naturally an elliptic function, we add something small to make it elliptic.  The $\widehat{\Psi}$-terms in the quotient below can remove this extra fudge factor in the aggregate quotient (where what remains \emph{is} elliptic), to obtain the quantity we are used to.

\begin{theorem}[Relation to $x$]
	\label{thm:relationtox}
	Let $\alpha$ and $\beta$ be non-zero isogenies from $E$ to $E'$ such that $\alpha \neq \pm \beta$.  Assume the characteristic is not two.
  Then
	\begin{equation}
		\label{eqn:diff}
		 \frac{\Psi_{\alpha + \beta} \Psi_{\alpha - \beta} \widehat{\Psi}_\alpha \widehat{\Psi}_\beta}{ \Psi_\alpha^2 \Psi_\beta^2  \widehat{\Psi}_{\alpha + \beta} } = x' \circ \beta - x' \circ \alpha.
	\end{equation}
\end{theorem}

\begin{proof}
  The parallelogram law is an integral quadratic identity:
	\[
		q(\alpha + \beta)  + q(\alpha - \beta) - 2q(\alpha) - 2q(\beta) + 2q(0) = 0.
	\]
  Thus from Lemma~\ref{lemma:quadrel} (observe that $\widehat{\Psi}_{\alpha + \beta} = \widehat{\Psi}_{\alpha - \beta}$), the left side has divisor
	\[
    K_{\alpha + \beta} + K_{\alpha - \beta} - 2 K_\alpha - 2 K_\beta.
	\]

  To see that this is also the divisor of the right hand side, let $\pi_1, \pi_2 : E' \times E' \rightarrow E'$ be the coordinate projections, and recall that the function $x \circ \pi_1 - x \circ \pi_2$ on $E' \times E'$ has divisor 
  \[
    \left( (\pi_1 + \pi_2)(P) = \mathcal{O}' \right)
+    \left( (\pi_1 - \pi_2)(P) = \mathcal{O}' \right)
    - 2\left( \pi_1(P) = \mathcal{O}' \right)
    - 2\left( \pi_2(P) = \mathcal{O}' \right).
  \]
Then, we pullback along $(\alpha, \beta) : E \rightarrow E' \times E'$ to obtain the divisor above.

	For the scaling on both sides of \eqref{eqn:diff}, we can compute the formal group expansions around $\mathcal{O}$.  
 For the left side, we have, using \eqref{eqn:norm} and \eqref{eqn:normhat}, that the leading coefficient is
 \[
    \frac{
      a_{\alpha-\beta}
      a_{\alpha+\beta}     }{
      a_\alpha^2 a_\beta^2 
    }.
 \]

  For the right side, first compute
  \begin{align*}
    (x' \circ \beta - x' \circ \alpha)(T) &= (a_\beta T^{\deg_{in}\beta})^{-2} - (a_\alpha T^{\deg_{in}\alpha})^{-2} + \cdots \\
    &= \frac{ (a_\alpha T^{\deg_{in}\alpha})^2-(a_\beta T^{\deg_{in}\beta})^2}{(a_\alpha T^{\deg_{in}\alpha})^2(a_\beta T^{\deg_{in}\beta})^2} + \cdots.
  \end{align*}
   Now, assume without loss of generality that $\deg_{in} \alpha \ge \deg_{in} \beta$.  
  Observe that 
\[
  a_{\alpha + \beta}T^{\deg_{in}(\alpha + \beta)} + \cdots = 
  \left( a_\alpha T^{\deg_{in}\alpha} + \cdots \right)
  +\left( a_\beta T^{\deg_{in}\beta} + \cdots \right).
\]
If $\deg_{in} \alpha = \deg_{in}\beta$, then comparing leading coefficients, $a_{\alpha \pm \beta} = a_\alpha \pm a_\beta$.  From this, we see that the leading coefficient on both sides of \eqref{eqn:diff} is $\frac{a_\alpha^2 - a_\beta^2}{a_\alpha^2 a_\beta^2} = a_\alpha^{-2} - a_\beta^{-2}$.

Otherwise, we have
\[
 a_{\alpha + \beta}T^{\deg_{in}(\alpha + \beta)} + \cdots = 
  a_\alpha T^{\deg_{in}\alpha} + \cdots ,
\]
from which $a_{\alpha + \beta} = a_\alpha$, and then both sides of \eqref{eqn:diff} have leading coefficient $a_\beta^{-2}$.
\end{proof}

The following alternate proof was suggested by David Grant, which shows that the biased recurrence relation pulls back under multiplication-by-2 to (and hence follows from) the unbiased case.  

\begin{proof}[Alternate proof]
Define
\[
  \Psi_\phi' := \frac{ \Psi_{2\phi} }{ \Psi_2^{\deg \phi} }
\]
having divisor
\[
    [2]^* \left( \phi^*(\mathcal{O}') - (\deg \phi) (\mathcal{O}) \right).
  \]
Define
\begin{equation}
  \label{eqn:2star}
  \Psis_\phi := \frac{ [2]^* \Psi_\phi }{ \Psi_\phi' },
\end{equation}
which from the Chain Rule (Theorem~\ref{thm:chain}) has the property that
\begin{equation}
  \label{eqn:2starB}
  \left( \Psis_\phi \right)^2 = [2]^*\widehat{\Psi}_\phi.
\end{equation}

  Now we can derive the result from the corresponding fact for unbiased isogenies (\cite[Appendix I, Proposition 1, Eqn. (4)]{MazurTate}), namely,
  \[
		 \frac{\Psi_{\alpha + \beta} \Psi_{\alpha - \beta} }{ \Psi_\alpha^2 \Psi_\beta^2 } = x' \circ \beta - x' \circ \alpha.
    \]
    To do this, 
    apply $[2]^*$ to \eqref{eqn:diff},
    and using \eqref{eqn:2star}, \eqref{eqn:2starB} and the observation that $\Psis_\alpha$ depends only on $\alpha$ modulo $2$, the statement is equivalent to
  \[
    \frac{\Psi'_{\alpha + \beta} \Psi'_{\alpha - \beta}
  }{ \left(\Psi'_\alpha \Psi'_\beta \right)^2 
  } 
		 = x' \circ 2\beta - x' \circ 2\alpha,
    \]
    which can be rewritten, using the fact that $\deg$ is quadratic, as 
    \[
      \frac{\Psi_{2\alpha + 2\beta} \Psi_{2\alpha - 2\beta} }
      { \left(\Psi_{2\alpha} \Psi_{2\beta} \right)^2  } 
		 = x' \circ 2\beta - x' \circ 2\alpha,
    \]
    which is an example of the unbiased case.
\end{proof}

\begin{corollary}[First recurrence relation]
	\label{cor:rec}
	For non-zero isogenies $\alpha, \beta, \gamma : E \rightarrow E'$, we have the recurrence
	\begin{equation}
		\label{eqn:bigrec}
		 \frac{\Psi_{\alpha + \beta} \Psi_{\alpha - \beta} \widehat{\Psi}_\alpha \widehat{\Psi}_\beta}{ \Psi_\alpha^2 \Psi_\beta^2  \widehat{\Psi}_{\alpha + \beta} } 
		 +\frac{\Psi_{\beta + \gamma} \Psi_{\beta - \gamma} \widehat{\Psi}_\beta \widehat{\Psi}_\gamma}{ \Psi_\beta^2 \Psi_\gamma^2  \widehat{\Psi}_{\beta + \gamma} } 
		 +\frac{\Psi_{\gamma + \alpha} \Psi_{\gamma - \alpha} \widehat{\Psi}_\gamma \widehat{\Psi}_\alpha}{ \Psi_\gamma^2 \Psi_\alpha^2  \widehat{\Psi}_{\gamma + \alpha} } 
		 =0.
	\end{equation}
\end{corollary}

\begin{remark}
  \label{rem:tilde}
Define another auxiliary kernel function $\widetilde\Psi_\phi$ with the principal divisor 
\[
  \widetilde{D}_\phi = 2 \phi^*(\mathcal{O}') {- 2 (\deg\phi) (\mathcal{O})}, 
\]
i.e., we require, as a means of normalization, that
\begin{equation*}
	\frac{ t^{2\deg \phi} \widetilde\Psi_\phi }{(t' \circ \phi)^2} (\mathcal{O}) 
	=  
	\left( \frac{dt}{\omega} (\mathcal{O}) \right)^{2\deg \phi} 
	\left( \frac{dt'}{\omega'} (\mathcal{O}') \right)^{-2}.
\end{equation*}
The notations $\widetilde\Psi_\phi$ and $\Psi_\phi$ agree with Satoh \cite[Definition 3.1]{Satoh} for unbiased endomorphisms, up to a factor of $\pm 1$.
In the case of endomorphisms for which all sums and differences of $\alpha,\beta,\gamma$ are unbiased, Satoh obtains \cite[Corollary 3.7]{Satoh} the special case
	\begin{equation*}
		 \frac{\Psi_{\alpha + \beta} \Psi_{\alpha - \beta} }{ \widetilde{\Psi}_{\alpha}\widetilde{\Psi}_\beta } 
		 +\frac{\Psi_{\beta + \gamma} \Psi_{\beta - \gamma} }{ \widetilde{\Psi}_\beta \widetilde{\Psi}_\gamma } 
		 +\frac{\Psi_{\gamma + \alpha} \Psi_{\gamma - \alpha}}{ \widetilde{\Psi}_\gamma \widetilde{\Psi}_\alpha } 
		 =0.
	\end{equation*}
Recall that $\widetilde{\Psi}_\alpha = \Psi_{\alpha}^2$ when $\alpha$ is unbiased, and $\widetilde{\Psi}_\alpha \widehat{\Psi}_\alpha = \Psi_\alpha^2$ in general; if all subscripts are unbiased this recovers the usual form \eqref{eqn:ellrec}.
\end{remark}

We now give the second, more general recurrence relation.  (Mazur and Tate do not address this one in the unbiased case.)

\begin{theorem}[Second relation to $x$]
	\label{thm:relationtox2}
	Let $\alpha$, $\beta$ and $\gamma$ be isogenies from $E$ to $E'$.  Then, when defined,
	\begin{equation}
    \label{eqn:zeta}
		\frac{
      \Psi_{\alpha + \beta + \gamma} \Psi_{\alpha - \beta} \Psi_{\gamma}
	 }{
		 {\Psi}_{\alpha+\gamma} {\Psi}_{\beta+\gamma} \Psi_\alpha \Psi_\beta 
	 } 
	 \sqrt{\frac{
		 \widehat{\Psi}_{\alpha+\gamma} \widehat{\Psi}_{\beta+\gamma} \widehat\Psi_\alpha \widehat\Psi_\beta 
	 }{
     \widehat\Psi_{\alpha + \beta + \gamma} \widehat\Psi_{\alpha - \beta}\widehat{\Psi}_\gamma
	 }}
   =
\frac{
     y' \circ \beta - y' \circ \gamma
   }{
     x' \circ \beta - x' \circ \gamma
   }
   -
 \frac{
     y' \circ \alpha - y' \circ \gamma
   }{
     x' \circ \alpha - x' \circ \gamma
   }
   .
	\end{equation}
\end{theorem}

\begin{proof}
  The cube law on three items $\alpha, \beta + \gamma, -\beta$ takes the form
  \[ 
    q(\alpha + \gamma) - q(\alpha - \beta) - q(\gamma) - q(\alpha +\beta + \gamma) + q(\alpha) + q(\beta + \gamma) + q(\beta) - q(0) = 0.
  \]
	Therefore the leftmost quotient is an instance of the cube law, so that Lemma~\ref{lemma:quadrel} applies.  Then the left side has divisor
	\[
    K_{\alpha + \beta + \gamma}
    + K_{\alpha - \beta} 
    - K_{\alpha + \gamma}
		- K_\alpha 
    - K_{\beta + \gamma}
		- K_\beta + K_\gamma.
	\]
  To see that this is the divisor of the right side, we again pullback from $E' \times E'$ as in the proof of Theorem~\ref{thm:relationtox}.  In this case, we know the divisor on $E' \times E'$ using classical formulas in complex analysis (see \cite[Lemma 3.6]{KateANT}):
  \[
    \frac{1}{2} \frac{ \wp'(a) - \wp'(c)}{\wp(a)-\wp(c)} - 
    \frac{1}{2} \frac{ \wp'(b) - \wp'(c)}{\wp(b)-\wp(c)} 
    =
    \zeta(a+c)-\zeta(a)-\zeta(b+c)+\zeta(b) = \frac{\sigma(a+b+c)\sigma(c)\sigma(a-b)}{\sigma(a+c)\sigma(b+c)\sigma(a)\sigma(b)}.
  \]

  Next, we check the constant.  
  By Lemma~\ref{lemma:squareroot} and applying \eqref{eqn:leading} to the normalization of the square root factor in \eqref{eqn:zeta}, the $T$-expansion of said square root factor has leading coefficient 
\[
    \frac{
     a_{\iota(\alpha+\gamma)} a_{\iota(\beta+\gamma)} a_{\iota(\alpha)} a_{\iota(\beta)} 
   }{
    a_{\iota(\alpha + \beta+ \gamma)} a_{\iota(\alpha - \beta)} a_{\iota(\gamma)}
   }.
\]
Applying that fact, together with \eqref{eqn:norm}, the leading coefficient of the left hand side of \eqref{eqn:zeta} is
\[
    \frac{
    a_{\alpha + \beta+ \gamma} a_{\alpha - \beta} a_{\gamma}
   }{
     a_{\alpha+\gamma} a_{\beta+\gamma} a_{\alpha} a_{\beta} 
   }.
\]
For the right hand side, assume for the moment that the inseparable degrees of $\alpha$, $\beta$ and $\gamma$ are equal.  Then, we consider the leading coefficients of the formal expansions around the origin, using \eqref{eqn:normalized-w}, and obtain
  \begin{align*}
 \frac{
     y' \circ \beta - y' \circ \gamma
   }{
     x' \circ \beta - x' \circ \gamma
   }
 -
 \frac{
     y' \circ \alpha - y' \circ \gamma
   }{
     x' \circ \alpha - x' \circ \gamma
   }
   &= 
   \frac{
     (a_\beta T^{\deg_{in}\beta})^{-3} - (a_\gamma T^{\deg_{in}\gamma})^{-3}
   }{
     (a_\beta T^{\deg_{in}\beta})^{-2} - (a_\gamma T^{\deg_{in}\gamma})^{-2}
   }
   - 
 \frac{
     (a_\alpha T^{\deg_{in}\alpha})^{-3} - (a_\gamma T^{\deg_{in}\gamma})^{-3}
   }{
     (a_\alpha T^{\deg_{in}\alpha})^{-2} - (a_\gamma T^{\deg_{in}\gamma})^{-2}
   } + \cdots.
 \end{align*} 
Consider the following algebraic identity in abstract variables $a$, $b$, $s$:
\begin{equation}
  \label{eqn:abstract}
\frac{
    b^{-3} - s^{-3}
  }{
    b^{-2} - s^{-2}
  }
- \frac{
    a^{-3} - s^{-3}
  }{
    a^{-2} - s^{-2}
  }
  =
  \frac{
    (a + b + s)(a - b)s
  }{
    (a + s)(b + s)ab
  }.
\end{equation}
Apply the identity \eqref{eqn:abstract} to $a,b,s = a_\alpha, a_\beta, a_\gamma$ and observe, as in the proof of Theorem~\ref{thm:relationtox}, that in the equal inseparable degree case, $a_{\alpha + \beta + \gamma} = a_\alpha + a_\beta + a_\gamma$, and similarly for other such sums.  This shows that the leading scalars on both sides of \eqref{eqn:zeta} agree.

For inseparable degrees that differ, we must consider the various cases in turn.
We may assume without loss of generality that $\deg_{in}\alpha \le \deg_{in}\beta$.

If $\deg_{in}\gamma < \deg_{in}\alpha < \deg_{in}\beta$, 
Again as in the proof of Theorem~\ref{thm:relationtox}, we have $a_{\alpha + \gamma} = a_\alpha$ etc. Using these simplifications,  
the leading term at left is $a_\beta^{-1}$ and at right is $a_\gamma a_\alpha a_\gamma / (a_\gamma^2 a_\alpha a_\beta )$.  
Then we use the identity $b^{-1} = sas/(s^2ab)$.  

If $\deg_{in}\alpha < \deg_{in}\gamma < \deg_{in}\beta$, we use the identity $b^{-1} = a^2s/(asab)$.  

If $\deg_{in}\gamma < \deg_{in}\alpha = \deg_{in}\beta$, we use the identity $b^{-1} - a^{-1} = s(a-b)s/(s^2ab)$.  

If $\deg_{in}\alpha < \deg_{in}\gamma = \deg_{in}\beta$, we use the identity $\frac{b^{-3}-s^{-3}}{b^{-2}-s^{-2}} - s^{-1} = a^2s/(a(b+s)ab)$.  

If $\deg_{in}\alpha = \deg_{in}\gamma < \deg_{in}\beta$, we use the identity $s^{-1} - \frac{a^{-3}-s^{-3}}{a^{-2}-s^{-2}} = -b^2s/((a+s)bab)$.  

If $\deg_{in}\alpha < \deg_{in}\beta < \deg_{in}\gamma$, then one expands the right side of \eqref{eqn:zeta} as
  \begin{equation*}
   \frac{
     (y' \circ \beta)(x' \circ \alpha) - (y' \circ \gamma)(x' \circ \alpha) - (y' \circ \beta)(x' \circ \gamma) - (y' \circ \alpha) (x' \circ \beta) + (y' \circ \gamma)(x' \circ \beta) + (y' \circ \alpha)(x' \circ \gamma)
   }{
     (x' \circ \beta - x' \circ \gamma)
     (x' \circ \alpha - x' \circ \gamma)
   }
 \end{equation*}
 and uses the identity $\frac{s^{-3}b^{-2}}{s^{-2}s^{-2}} = \frac{a^2s}{abab}$.

If $\deg_{in}\alpha = \deg_{in}\beta < \deg_{in}\gamma$, then with this same expansion, one uses
$\frac{-s^{-3}a^{-2} + s^{-3}b^{-2}}{s^{-2}s^{-2}} = \frac{(a+b)(a-b)s}{abab}$.
\end{proof}

The following is an immediate consequence.

\begin{corollary}[Second recurrence relation]
	\label{cor:rec2}
  Wherever the following terms are defined
  \begin{equation}
    \begin{split}
    \label{eqn:bigrec2}
		\frac{
		 \Psi_{\alpha + \beta + \sigma} \Psi_{\alpha - \beta}		
	 }{
		 {\Psi}_{\alpha+\sigma} {\Psi}_{\beta+\sigma} \Psi_\alpha \Psi_\beta 
	 } 
	 \sqrt{\frac{
		 \widehat{\Psi}_{\alpha+\sigma} \widehat{\Psi}_{\beta+\sigma} \widehat\Psi_\alpha \widehat\Psi_\beta 
	 }{
		 \widehat\Psi_{\alpha + \beta + \sigma} \widehat\Psi_{\alpha - \beta}		
	 }}
   &+
	\frac{
		 \Psi_{\beta + \gamma + \sigma} \Psi_{\beta - \gamma}	
	 }{
		 {\Psi}_{\beta+\sigma} {\Psi}_{\gamma+\sigma} \Psi_\beta \Psi_\gamma 
	 } 
	 \sqrt{\frac{
		 \widehat{\Psi}_{\beta+\sigma} \widehat{\Psi}_{\gamma+\sigma} \widehat\Psi_\beta \widehat\Psi_\gamma 
	 }{
		 \widehat\Psi_{\beta + \gamma + \sigma} \widehat\Psi_{\beta - \gamma}	
	 }} \\
   &+
	\frac{
		\Psi_{\gamma + \alpha + \sigma} \Psi_{\gamma - \alpha}
	 }{
		 {\Psi}_{\gamma+\sigma} {\Psi}_{\alpha+\sigma} \Psi_\gamma \Psi_\alpha 
	 }
	 \sqrt{\frac{
		 \widehat{\Psi}_{\gamma+\sigma} \widehat{\Psi}_{\alpha+\sigma} \widehat\Psi_\gamma \widehat\Psi_\alpha 
	 }{
		 \widehat\Psi_{\gamma + \alpha + \sigma} \widehat\Psi_{\gamma - \alpha}	
	 }}
=0.
\end{split}
	\end{equation}
	In particular, if all the indices represent unbiased isogenies, then the $\Psi_\phi$ satisfy \eqref{eqn:ellrec-gen}.
\end{corollary}

In order to compare this statement with Corollary~\ref{cor:rec}, observe that
\[
  \frac{ \widehat{\Psi}_\alpha \widehat{\Psi}_\beta }{ \widehat{\Psi}_{\alpha + \beta} } =
  \sqrt{
    \frac{ \widehat{\Psi}_\alpha^2 \widehat{\Psi}_\beta^2 }
    { \widehat{\Psi}_{\alpha + \beta} \widehat{\Psi}_{\alpha - \beta}}
  }.
\]

\section{Specialization and elliptic nets}

Specializing (evaluating) $\Psi_\phi$ at a specific point $P \in E$, we obtain a sequence of values of the field satisfying the recurrence relations \eqref{eqn:bigrec} and \eqref{eqn:bigrec2}.  
Observe that we have a choice of $\omega_i$.  Fix a particular point $P \in E$.  We can choose $\omega_i$, $i = 1,2,3$ so that $\widehat{\Psi}_i(P) = 1$; by our convention \eqref{eqn:convention}, this dictates the choice of $\omega = \omega_0$.  
In this case, the extra factors of $\widehat{\Psi}$ disappear and we recover the usual recurrences \eqref{eqn:ellrec} and \eqref{eqn:ellrec-gen} for that value of $P$.  We cannot, however, choose such a normalization globally (that is, simultaneously for all $P$).

\begin{definition}
Fix a point $P \in E$.  The collection $\Psi_\phi(P)$, as $\phi$ varies, is called \emph{consonant} if the $\omega_i$ are chosen so that $\widehat{\Psi}_i(P) = 1$ for all $i=1,2,3$.
\end{definition}

\begin{lemma}
  A consonant collection $\Psi_\phi(P)$ satisfies \eqref{eqn:ellrec} and \eqref{eqn:ellrec-gen}.
\end{lemma}

Next we recall some general results classifying collections satisfying \eqref{eqn:ellrec-gen}.

Net polynomials \cite{KateANT} generalize division polynomials.  Define for any vector $\vec a = (a_1,\ldots, a_k) \in \ZZ^k$, an elliptic function $\Psi_{\vec a}$ on $E^k$ with divisor
\begin{equation}
	\label{eqn:phidiv}
	\left(\sum a_i P_i = \mathcal{O}\right)
	- \sum_i a^2_i ({P}_i = \mathcal{O})
	- \sum_{i < j} a_ia_j \left( ({P}_i + {P}_j = \mathcal{O}) - (P_i = \mathcal{O}) - (P_j = \mathcal{O}) \right),
\end{equation}
and normalized in a manner similar to the previous cases, namely, where we denote by $\sigma$ the summation function $(P_1,\ldots,P_k) \mapsto P_1 + \cdots + P_k$, and by $\pi_i$ the projection onto the $i$-th component, and require
\begin{equation}
	\label{eqn:norm-arb-gen}
	\frac{ 
	\Psi_{\vec a} 
	\prod_i (t^{a_i^2 - \sum_{j \neq i} a_ia_j} \circ \pi_i) 
	\prod_{i < j} t^{a_ia_j} \circ (\sigma \circ (\pi_i \times \pi_j)) 
	}{
		t \circ \sigma \circ (a_1\times \cdots \times a_k)
	} (\mathcal{O}) 
	= \left( \frac{dt}{\omega} (\mathcal{O}) \right)^{\sum_i a_i^2 - \sum_{i < j} a_i a_j - 1}.
\end{equation}
The means of normalizing in \cite{KateANT} differs, but amounts to the same thing: in both means of normalizing, the dependence on $\omega$ is the same \cite[Proposition 7.1]{KateANT}, and $\Psi_{\mathbf{e}_i}(P) = 1$ for the standard basis vectors $\mathbf{e}_i$.

Interpreting the indices of \eqref{eqn:ellrec} and \eqref{eqn:ellrec-gen} as elements of $\ZZ^k$, the $\Psi_{\vec a}$ satisfy both recurrences \cite[Theorem 4.1]{KateANT}.  More generally, we call any $k$-dimensional array which satisfies \eqref{eqn:ellrec-gen} an \emph{elliptic net} \cite[Definition 1.1]{KateANT}.

The $\psi_{\vec a}$ also satisfy the usual relationship to $x$ \cite[Lemma 4.2]{KateANT}, and a version of the chain rule \cite[Proposition 4.3]{KateANT}, namely, for a $k \times k$ linear transformation $T$, and standard basis vectors $\mathbf{e}_i$,
\[
	(\Psi_{\vec a} \circ T)
	\prod_{i=1}^k \Psi_{T^{tr}(\mathbf{e}_i)}^{a_i^2 - \sum_{j \neq i} a_i a_j}
	\prod_{i < j} \Psi_{T^{tr}(\mathbf{e}_i+\mathbf{e}_j)}^{a_ia_j}
	=
	\Psi_{T^{tr}(\vec a)}.
\]

Ward's theorem classifying elliptic divisibility sequences extends to elliptic nets \cite[Theorem 7.4]{KateANT}.  This result states that, up to appropriate equivalences and normalizations and degenerate cases, $n$-dimensional arrays satisfying \eqref{eqn:ellrec-gen} are in bijection with tuples $(E, P_1, \ldots, P_n)$.  In particular, since for a fixed point $P$, the $\Psi_{\phi_i}(P)$ satisfy \eqref{eqn:ellrec-gen} (after suitable normalization), we can conclude that they form an elliptic net of the form $\Psi_{\vec a}(P_1, \ldots, P_k)$ for some choice of curve $E$ and points $P_i \in E$ (we say it is the elliptic net \emph{associated to} this choice of curve and points).  The following theorem verifies this constructively.

\begin{theorem}
  \label{thm:recover}
	Let $\phi_1, \ldots, \phi_n \in \Hom(E,E')$.  Let $P \in E$.  Suppose that $\phi_i(P) \notin E'[2]$ and $(\phi_i \pm \phi_j)(P) \neq 0$ for all $i,j$.
	Choose the $\omega_i$ so that the resulting division polynomials $W(\vec a) := \Psi_{\sum a_i \phi_i}(P)$ form a consonant collection.  Then they form an elliptic net associated to $E'$ and the points $\phi_i(P)$.
\end{theorem}

\begin{proof}
	This is an application of the chain rule.  In particular, since the collection is consonant, it forms an elliptic net in the sense that it satisfies \eqref{eqn:ellrec-gen}, where the indices $\vec a$ are interpreted as $\sum_i a_i \phi_i$.  To apply \cite[Theorem 7.4]{KateANT}, we require that the elliptic net be non-degenerate.  That is, we require $\Psi_{\phi_i}(P)$, $\Psi_{2\phi_i}(P)$, $\Psi_{\phi_i + \phi_j}(P)$ and $\Psi_{\phi_i - \phi_j}(P)$ to be non-zero.  This is guaranteed by the hypotheses.

  To figure out which curve and points this elliptic net represents, we can look at elliptic divisibility sequences $\Psi_{n\phi_i}(P)$, $n \in \ZZ$ for each $i$.  Fixing $i$, by the chain rule (Theorem~\ref{thm:chain}), using the assumption the collection is consonant,
	\[
		\Psi_{n \phi_i}(P)
		= \Psi_{n}(\phi_i(P)) \Psi_{\phi_i}(P)^{n^2}.
	\]
	Thus, up to normalization, the associated curve and point are $E'$ and $\phi_i(P)$.  From this we conclude that the elliptic net is that associated to $(E', \phi_1(P), \ldots, \phi_n(P))$.
\end{proof}

This shows that the collection of division polynomials $\Psi_\alpha$ for $\alpha \in \End(E)$ (differentials suitably normalized) form an elliptic net whose rank is equal to the rank of the endomorphism ring $\End(E)$.  In particular, by the classification theorem for the endomorphism ring of an elliptic curve, these are equivalent to an elliptic net associated to a single, pair or quadruple of points.

The term \emph{magnified} has been applied to elliptic divisibility sequences associated to image points of rational isogenies; see for example \cite{Magnified}.  The theorem can be read as saying that the collections $\Psi_\phi(P)$ can be interpreted as `magnified elliptic nets' of rank one, two or four.

\section{Examples}
\label{sec:examples}

We will compute generalized division polynomials for Gaussian CM (comparing to Ward's case), for an isogeny, and for a supersingular curve with inseparability.  We use these to verify examples of the chain rule, recurrence relation, and relation to $x$.  

\subsection{An elliptic curve with CM by $\ZZ[i]$}

Let $E: y^2 = x^3 - x$, which is an elliptic curve over $\QQ$ with complex multiplication by $\ZZ[i]$.  Use the usual invariant differential $\omega = dx/2y$, and uniformizer $T = -x/y$.  We have $[i]:(x,y) \mapsto (-x,yi)$ and $E[2] = \{ (0,0), (1,0), (-1,0) \}$.
We also have
\begin{align*}
  \Psi_1 &= 1, \quad
  \Psi_2 = -2y, \quad
  \Psi_3 = 3\left(x^4 - 2x^2 - \frac{1}{3}\right), \\
  \Psi_4 &= -4y(x+i)(x-i)(x^2 - 2x - 1)(x^2 + 2x-1), \\
  \Psi_5 &= 5\left(x^2 - \frac{2}{5}i - \frac{1}{5}\right)\left(x^2 + \frac{2}{5}i - \frac{1}{5}\right)\left(x^8 - 12x^6 - 26x^4 + 52x^2 + 1\right).
\end{align*}
Now, $\Psi_i$ is constant (begin unbiased with trivial kernel), and to determine the constant, we have the requirement that
\[
  \frac{T \Psi_i}{T \circ [i]}(\mathcal{O}) = 1,
\]
which implies that $\Psi_i = i$.
Similarly (see Section~\ref{sec:insep} for more details on such computations in more challenging cases), we can compute
\begin{align*}
  \Psi_i &= i, \quad
  \Psi_{1+i} = 2ix, \quad
  \Psi_{1-i} = 2x, \\
  \Psi_{1+2i} &= (1+2i)\left(x^2 + \frac{2}{5}i - \frac{1}{5}\right) = (1+2i)x^2 - 1, \\
  \Psi_{1-2i} &= (1-2i)\left(x^2 - \frac{2}{5}i - \frac{1}{5}\right) = (1-2i)x^2 -1, \\
  \Psi_{2+i} &= (2+i)\left(x^2 - \frac{2}{5}i - \frac{1}{5}\right) = (2+i)x^2 - i, \\
  \Psi_{2-i} &= (2-i)\left(x^2 + \frac{2}{5}i - \frac{1}{5}\right) = (2-i)x^2 + i.
\end{align*}

Next we verify the recurrence relation in one case.
Let $\alpha = 1+i$, $\beta = i$, $\gamma = 1$.  Then we have
\begin{align*}
		 \frac{\Psi_{\alpha + \beta} \Psi_{\alpha - \beta} \widehat{\Psi}_\alpha \widehat{\Psi}_\beta}{ \Psi_\alpha^2 \Psi_\beta^2  \widehat{\Psi}_{\alpha + \beta} } 
     &=
     \frac{ \Psi_{1+2i} \Psi_{1}\widehat{\Psi}_{1+i} \widehat{\Psi}_i}
     { \Psi_{1+i}^2 \Psi_{i}^2 \widehat{\Psi}_{1+2i}} 
     =
     \frac{\left( (1+2i)x^2-1 \right)2ix }
     {(2i)^2x^2i^2};\\
		 \frac{\Psi_{\beta + \gamma} \Psi_{\beta - \gamma} \widehat{\Psi}_\beta \widehat{\Psi}_\gamma}{ \Psi_\beta^2 \Psi_\gamma^2  \widehat{\Psi}_{\beta + \gamma} } 
     &=
     \frac{ \Psi_{1+i} \Psi_{i-1}\widehat{\Psi}_{i} \widehat{\Psi}_1}
     { \Psi_{i}^2 \Psi_{1}^2 \widehat{\Psi}_{1+i}} 
     =
     \frac{(2i)x(-2)x}{i^2(2i)x};
     \\
		 \frac{\Psi_{\gamma + \alpha} \Psi_{\gamma - \alpha} \widehat{\Psi}_\gamma \widehat{\Psi}_\alpha}{ \Psi_\gamma^2 \Psi_\alpha^2  \widehat{\Psi}_{\gamma + \alpha} } 
     &=
     \frac{ \Psi_{2+i} \Psi_{-i}\widehat{\Psi}_{1} \widehat{\Psi}_{1+i}}
     { \Psi_{1}^2 \Psi_{1+i}^2 \widehat{\Psi}_{2+i}}
     =
     \frac{
\left( (2+i)x^2 - i \right)
(-i)(2i)x}{(2i)^2x^2 }.
\end{align*}
Plugging these in verifies the first recurrence relation (Corollary~\ref{cor:rec}) in this example.

\subsection{Comparison with Ward}
It is worth comparing to Ward's $\ZZ[i]$ case \cite{WardGaussian}.  He obtained polynomials 
  \begin{align*}
    P_1 &= 1, \quad
    P_2 = 2, \\
  P_i &= i, \quad
  P_{1+i} = 1+i, \quad
  P_{1-i} = 1-i, \\
  P_{1+2i} &= (1+2i)\left(x^2 + \frac{2}{5}i - \frac{1}{5}\right) = (1+2i)x^2 - 1, \\
  P_{3+i} &= (3+i)x^4 - 2(1+3i)x^2 + (3+i).
  \end{align*}
  His polynomials do not satisfy the usual recurrence \eqref{eqn:ellrec}; instead the polynomials $E_\alpha P_\alpha$ do, where $E_\alpha = 1$, $\sqrt{x}$ or $y$, according as\footnote{Ward referred to these three cases as $\alpha$ being \emph{odd}, \emph{oddly even} or \emph{totally even}.} the norm of $\alpha$ is $1, 2$, or $0$ modulo $4$.  We have $E_\alpha P_\alpha = \Psi_\alpha$ in the unbiased case.

  Let $\alpha = 1+i$, $\beta = i$, $\gamma = 1$.  Then setting $\overline{\Psi}_\alpha = E_\alpha P_\alpha$, we can verify the same instance of the recurrence relation by adding the following terms:
\begin{align*}
  \frac{\overline{\Psi}_{\alpha + \beta} \overline{\Psi}_{\alpha - \beta}}{ \overline{\Psi}_{\alpha}^2 \overline{\Psi}_{\beta}^2}
     &=
   \frac{ \overline{\Psi}_{1+2i} \overline{\Psi}_{1}}
 { \overline{\Psi}_{1+i}^2 \overline{\Psi}_{i}^2 } 
     =
     \frac{(1+2i)x^2 - 1}
     {(1+i)^2i^2 x};\\
     \frac{\overline{\Psi}_{\beta + \gamma} \overline{\Psi}_{\beta - \gamma}}{ \overline{\Psi}_{\beta}^2 \overline{\Psi}_{\gamma}^2 } 
     &=
   \frac{ \overline{\Psi}_{1+i} \overline{\Psi}_{i-1}}
 { \overline{\Psi}_{i}^2 \overline{\Psi}_{1}^2} 
     =
     \frac{(1+i)(1-i)x}{i^2};
     \\
     \frac{\overline{\Psi}_{\gamma + \alpha} \overline{\Psi}_{\gamma - \alpha}}{ \overline{\Psi}_{\gamma}^2 \overline{\Psi}_{\alpha}^2 } 
     &=
   \frac{ \overline{\Psi}_{2+i} \overline{\Psi}_{i}}
   { \overline{\Psi}_{1}^2 \overline{\Psi}_{1+i}^2}
     =
     \frac{
       ((2+i)x^2-i)
(-i)}{(1+i)^2x}.
\end{align*}
Ward's use of $\sqrt{x}$ was justified by his restriction to the complex analytic case.

\subsection{An isogeny division polynomial}

There is a rational isogeny from $E$ to $E' : y^2 = x^3 -11x - 14$ given by
\[
  \phi_{(1,0)}: (x,y) \mapsto \left( \frac{x^2 - x + 2}{x - 1}, \frac{y(x^2 -2x - 1)}{x^2 -2x + 1} \right)
\]
with kernel $\{ (1,0), \mathcal{O} \}$.  To compute the associated division polynomial, observe that $x-1$ has the correct divisor.  To set the normalization, choose the normalized invariant differential on the target curve (so $dT/\omega = 1 + O(T)$ as in Section~\ref{sec:recmain}) and compute
\[
\phi T = -\frac{x'}{y'} = - \frac{ (x^2 - x + 2)(x^2 - 2x + 1) }{ y(x-1)(x^2 - 2x - 1)} 
= - \frac{x}{y}\frac{ T^{-6} + \cdots}{T^{-6} + \cdots} = T + \cdots.
\]
So we obtain 
\[
\Psi_{\phi} = \widehat{\Psi}_\phi = x-1.
\]
The kernel of $\phi \circ (1+i)$ is cyclic of order $4$: 
\[
  \{ \mathcal{O}, (i, i-1), (i, -i+1), (0,0) \}.
\]
This has kernel sum $(0,0)$. 
Therefore, up to a scalar, the generalized division polynomial is 
\[
  x(x-i).
\]
Combining the known scalars for $\phi$ and $1+i$, we have
\[
  \Psi_{\phi \circ (1+i)} = 2ix(x-i).
\]
We also have 
\[
  \widehat{\Psi}_\phi \circ (1+i) = x \circ [1+i] - 1 = \frac{-i(x^2-1)}{2x} -1 = \frac{(x-i)^2}{2ix}, \quad
  \widehat{\Psi}_{\phi \circ (1+i)} = \widehat{\Psi}_{1+i} = 2ix.
\]
We verify the chain rule (Theorem~\ref{thm:chain}) by checking
\[
  \left(
  \frac{
  {\Psi}_{\phi \circ (1+i)}
}{
  ({\Psi}_\phi \circ (1+i)) {\Psi}_{1+i}^2
}\right)^2
= \frac{1}{(x-i)^2}
=  \frac{
  \widehat{\Psi}_{\phi \circ (1+i)}
}{
  (\widehat{\Psi}_\phi \circ (1+i)) \widehat{\Psi}_{1+i}^2
}.
\]

\subsection{A supersingular example with inseparability}
\label{sec:insep}

Consider the supersingular curve $E'': y^2 = x^3 + x$ over $\FF_7$.  
It has $E''[2] = \{ \mathcal{O}, (0,0), (I,0), (-I,0) \}$, where $I^2 + 1 = 0$.  
Furthermore, a basis for the $4$-torsion is given by $P = (4I, 4I+3)$ and $Q = (2I, 5I+5)$ where $2P = (I,0)$ and $2Q = (-I,0)$.  
A basis for the $3$-torsion has $x$-coordinates $\pm 1+4I$.
We can compute some classical division polynomials:
\begin{align*}
  \Psi_1 &= 1, \quad \Psi_2 = 2y, \quad \Psi_3 = 3x^4 - x^2 - 1 = 3 (x + 3I + 1)  (x + 3I + 6)  (x + 4I + 1)  (x + 4I + 6),\\
  \Psi_4 &= y(-3x^6 - x^4 + x^2 + 3) = 4y(x+1)(x-1)(x-3I)(x+3I)(x-2I)(x+2I), \\
  \Psi_5 &= -2x^{12} - x^{10} + x^6 + x^4 - x^2 + 1, \\
  \Psi_6 &= y(-x^{16} - 3x^{14} - x^8 - 3x^2 - 1), \\ 
  \Psi_7 &= - 1.
\end{align*}
Write $\FF_{49} = \FF_7[I]$ where $I^2 =-1$.  
The curve $E''$ has endomorphism ring $\ZZ + i\ZZ + \frac{1+k}{2} \ZZ + \frac{i+j}{2} \ZZ$, where $i: (x,y) \mapsto (-x, Iy)$ and $j: (x,y) \mapsto (x^7, y^7)$ is the Frobenius endomorphism, and $k=ij$.  

In order to compute further division polynomials, we begin with determining the leading coefficients of the basis elements of the endomorphism ring.
We use the normalized invariant differential and have
\[
  i T = \frac{x}{Iy} = I \left(-\frac{x}{y}\right).
\]
So we have $\Psi_i = I$. 
Next
\begin{align*}
  (i\pm j) T &= - \frac{ {\pm y^7 - Iy} + x^2 - x^{14} }{ - \frac{\pm y^7-Iy}{x^7+x}\left( {\pm y^7 - Iy} + x^2 - x^{14} \right) - Iyx^7 \mp y^7x   } =  \frac{ -Iyx }{y^2} =  I T + \cdots 
\end{align*}
Also $[2]T = 2T + \cdots$.  Therefore, we have $\frac{i\pm j}{2} T = -3I T + \cdots$.  In particular, then, $\Psi_{\frac{i\pm j}{2}} = -9T + \cdots$.  The reason for the square is \eqref{eqn:normhat}, since this endomorphism is biased and the divisor for its division polynomial is $2(P_i)-2(\mathcal{O})$ for some $2$-torsion point $P_i$.

Observe that the minimal polynomial of $\frac{\pm i \pm j}{2}$ is $x^2 + 2$ having roots $\pm 3I$ in $\FF_{49}$, which is a check against the calculation above (we expect the leading coefficient of the $T$-expansion to satisfy the minimal polynomial).  Since $\frac{-1\pm k}{2} = i \frac{i\pm j}{2}$, the leading coefficient of the $T$-expansion of $\frac{-1\pm k}{2}$ is $3$, and that of $\frac{1\pm k}{2}$ is $-3$.  When everything is inseparable, leading coefficients will add (by the theory of formal groups), so that for example $\frac{1+i+j+k}{2} = 1 + \frac{-1+k}{2} + \frac{i+j}{2}$ has leading coefficient $1 + 3 -3I = -3-3I$.  The minimal polynomial is $x^2 -x +4$ having roots $-3 \pm 3I$.  In this way, we determine all our leading coefficients for the following division polynomials.

First, we consider endomorphisms of degree $2$, which are biased.  One can verify directly that $\frac{i+j}{2}(-I,0) = (i+j)(2I, \pm(5I+5)) = \mathcal{O}$.  Since $\frac{i+j}{2}$ must have just one primitive two-torsion point in the kernel, we find that
$\Psi_{\frac{i+j}{2}} = -9(x+I)$ (recall the divisor for biased division polynomials \eqref{eqn:Dphi}).  Similar reasoning obtains the other cases $\frac{\pm i \pm j}{2}$.
Observe that $i(i+j) = -1 + k$, so this also gives polynomials $\Psi_{\frac{\pm 1 \pm k}{2}}$, which are other unit multiples.

Next, endomorphisms of degree $3$, which are unbiased.  A computation reveals that $3I+6$ is the $x$-coordinate of the $3$-torsion point in the kernel of $1-i-j$, which acts exactly as $1 + \frac{i+j}{2}$ on the $3$-torsion.  Thus $\Psi_{1 + \frac{i+j}{2}}$ is $(1-3I)(x + 1 - 3I)$.  Observe that the minimal polynomial in this case is $x^2 - 2x + 3$ which has roots $1 \pm 3I$.  Similar computations give the division polynomials for other endomorphisms of degree $3$.

Finally, we consider cyclic endomorphisms of degree $4$, which are again biased.  
It will be helpful to compute, using V\'elu's formulas \cite{Velu}, the explicit isogenies
\begin{equation}
\begin{aligned}
  (1+i)(x,y) &= \left( 3I\frac{x^2+1}{x}, \frac{s^2y - y}{x^2} \right), \\
  \frac{i-j}{2}(x,y) &= \left( 3\frac{x^2-Ix-2}{x-I}, I\frac{s^2y +2Ixy + y}{x^2+2Ix-1} \right), \\
  \frac{i+j}{2}(x,y) &= \left( 3\frac{x^2+Ix-2}{x+I}, I\frac{s^2y -2Ixy + y}{x^2-2Ix-1} \right).
\end{aligned}
  \label{eqn:explicitiso}
\end{equation}

For the case of $\Psi_{\frac{3i \pm j}{2}}$, it must divide $\Psi_4$ and have a cyclic subgroup of $E[4]$ as the endomorphism kernel, since the minimal polynomial of $\frac{3i \pm j}{2}$ is $x^2 + 4$, and its leading coefficient is $I - 3I = -2I$.  We can compute its action on the $4$-torsion explicitly using the formulas above. 
The leading coefficient of the $T$-expansion of $\frac{3+k}{2}$ is $-2$, which satisfies the minimal polynomial $x^2 - 3x - 3 = (x+2)^2$.

The same type of computation determines $\frac{1+i+j+k}{2}$ etc.

With this, we can compute some generalized division polynomials:
\begin{align*}
  \Psi_{i} &= I, \quad \Psi_{j} = 1,  \quad \Psi_{1+i} = (1+I)^2x, \\
  \Psi_{\frac{i+j}{2}} &= -9(x+I), \quad
  \Psi_{\frac{i-j}{2}} = -9(x-I), \quad
  \Psi_{\frac{1-k}{2}} = -9I(x+I), \quad
  \Psi_{\frac{1+k}{2}} = -9I(x-I), \\
  \Psi_{1 + \frac{i+j}{2}} &= (1-3I)(x+1-3I), \quad
  \Psi_{1 - \frac{i+j}{2}} = (1+3I)(x-1-3I), \\
\Psi_{\frac{1+i+j+k}{2}} &=  (-3-3I)(1+I)x(x-1), \quad
\Psi_{\frac{-1+i+j-k}{2}} =  (3-3I)(1+I)x(x+1), \\
\Psi_{\frac{3+k}{2}} &=  -2(-3I)(x+I)(x-2I),
\Psi_{\frac{3i-j}{2}} = I \Psi_{\frac{3+k}{2}} = -2I(-3I)(x+I)(x-2I), \\
\Psi_{\frac{3i+j}{2}} &= -2I(-3I)(x-I)(x+2I). 
\end{align*}

For computation of $\widehat{\Psi}$, we fix the three $2$-isogenies $g_1 = 1+i$, $g_2 = \frac{i+j}{2}$, and $g_3 = \frac{i-j}{2}$.  
Now we can verify the recurrence relation with $\alpha = \frac{i+j}{2}$, $\beta = \frac{1+k}{2}$ and $\gamma = 1$:
\begin{align*}
		 \frac{\Psi_{\alpha + \beta} \Psi_{\alpha - \beta} \widehat{\Psi}_\alpha \widehat{\Psi}_\beta}{ \Psi_\alpha^2 \Psi_\beta^2  \widehat{\Psi}_{\alpha + \beta} } 
     &=
     \frac{ \Psi_{\frac{1+i+j+k}{2}} \Psi_{\frac{-1+i+j-k}{2}}\widehat{\Psi}_{\frac{i+j}{2}} \widehat{\Psi}_{\frac{1+k}{2}}}
     { \Psi_{\frac{i+j}{2}}^2 \Psi_{\frac{1+k}{2}}^2 \widehat{\Psi}_{\frac{1+i+j+k}{2}}  } \\
     &=
     \frac{(-3-3I)(1+I)x(x-1)(3-3I)(1+I)x(x+1)(-9)(x+I)(-9)(x-I)}
     {(-9)^2(x+I)^2 (-9I)^2(x-I)^2(1+I)^2x};\\
     &= \frac{x(x+1)(x+1)}{(x-I)(x-I)} \\
		 \frac{\Psi_{\beta + \gamma} \Psi_{\beta - \gamma} \widehat{\Psi}_\beta \widehat{\Psi}_\gamma}{ \Psi_\beta^2 \Psi_\gamma^2  \widehat{\Psi}_{\beta + \gamma} } 
     &=
      \frac{ \Psi_{\frac{3+k}{2}} \Psi_{\frac{-1+k}{2}}\widehat{\Psi}_{\frac{1+k}{2}} \widehat{\Psi}_1}
     { \Psi_{\frac{1+k}{2}}^2 \Psi_{1}^2 \widehat{\Psi}_{\frac{3+k}{2}}  }
     =
     \frac{-2(-3I)(x+I)(x-2I)(-9I)(x+I)(-9)(x-I)}
     {(-9I)^2(x-I)^2 (-9)(x+I)};\\
     &= 4 \frac{(x-2I)(x+I)}{(x-I)} \\
		 \frac{\Psi_{\gamma + \alpha} \Psi_{\gamma - \alpha} \widehat{\Psi}_\gamma \widehat{\Psi}_\alpha}{ \Psi_\gamma^2 \Psi_\alpha^2  \widehat{\Psi}_{\gamma + \alpha} } 
     &=
     \frac{ \Psi_{1+\frac{i+j}{2}} \Psi_{1-\frac{i+j}{2}}\widehat{\Psi}_{1} \widehat{\Psi}_{\frac{i+j}{2}}}
     { \Psi_{1}^2 \Psi_{\frac{i+j}{2}}^2 \widehat{\Psi}_{1+\frac{i+j}{2}}  }
     =
     \frac{(1-3I)(x+1-3I)(1+3I)(x-1-3I)(-9)(x+I)}
     {(-9)^2(x+I)^2} \\
     &= 2\frac{(x+1-3I)(x-1-3I)}{(x+I)}
\end{align*}
These three terms add to zero.

Now we can verify the recurrence in an inseparable case, with $\alpha = \frac{i+j}{2}$, $\beta = i$ and $\gamma = \frac{i-j}{2}$:
\begin{align*}
		 \frac{\Psi_{\alpha + \beta} \Psi_{\alpha - \beta} \widehat{\Psi}_\alpha \widehat{\Psi}_\beta}{ \Psi_\alpha^2 \Psi_\beta^2  \widehat{\Psi}_{\alpha + \beta} } 
     &=
     \frac{ \Psi_{\frac{3i+j}{2}} \Psi_{\frac{-i+j}{2}}\widehat{\Psi}_{\frac{i+j}{2}} \widehat{\Psi}_i}
     { \Psi_{\frac{i+j}{2}}^2 \Psi_{i}^2 \widehat{\Psi}_{\frac{3i+j}{2}}  }
     =
     \frac{-2I(-3I)(x-I)(x+2I)(9)(x-I)(-9)(x+I)}
     {(-9)^2(x+I)^2(I)^2(-2I)(-9)(x-I)};\\
     &= 3 \frac{(x-I)(x+2I)}{(x+I)} \\
		 \frac{\Psi_{\beta + \gamma} \Psi_{\beta - \gamma} \widehat{\Psi}_\beta \widehat{\Psi}_\gamma}{ \Psi_\beta^2 \Psi_\gamma^2  \widehat{\Psi}_{\beta + \gamma} } 
     &=
      \frac{ \Psi_{\frac{3i-j}{2}} \Psi_{\frac{i+j}{2}}\widehat{\Psi}_i \widehat{\Psi}_{\frac{i-j}{2}} }
     { \Psi_{i}^2 \Psi_{\frac{i-j}{2}}^2 \widehat{\Psi}_{\frac{3i-j}{2}}  }
     =
     \frac{-2I(-3I)(x+I)(x-2I)(-9)(x+I)(-9)(x-I)}
     {(-9)^2(x-I)^2(I)^2(-9)(x+I)};\\
     &= -3\frac{(x+I)(x-2I)}{(x-I)}, \\
		 \frac{\Psi_{\gamma + \alpha} \Psi_{\gamma - \alpha} \widehat{\Psi}_\gamma \widehat{\Psi}_\alpha}{ \Psi_\gamma^2 \Psi_\alpha^2  \widehat{\Psi}_{\gamma + \alpha} } 
     &=
     \frac{ \Psi_{i} \Psi_{-j} \widehat{\Psi}_{\frac{i-j}{2}} \widehat{\Psi}_{\frac{i+j}{2}}}
     { \Psi_{\frac{i-j}{2}}^2 \Psi_{\frac{i+j}{2}}^2 \widehat{\Psi}_{i}  }
     =
     \frac{(I)(-1)(-9)(x-I)(-9)(x+I)}
     {(-9)^2(x-I)^2(-9)^2(x+I)^2} \\
     &= -2I \frac{1}{(x-I)(x+I)}. 
\end{align*}
This sums to zero also.  
To double check the simplification of the third term, we can use the divisor:
\begin{align*}
  K_i + K_{j} - 2K_{\frac{i+j}{2}} - 2K_{\frac{i-j}{2}} &= (\mathcal{O}) + 7(\mathcal{O}) - 2( ((-I,0)) + (\mathcal{O}) )- 2( ((I,0)) + (\mathcal{O}) ) \\
                                                         &=   4(\mathcal{O}) -2((-I,0)) - 2((I,0)) \\
                                                         &= \div\left( \frac{1}{(x-I)(x+I)} \right).
\end{align*}
Finally, the third term, together with \eqref{eqn:explicitiso}, can be used to verify the relation to $x$ (Theorem~\ref{thm:relationtox}):
\[
 \frac{ \Psi_{i} \Psi_{j} \widehat{\Psi}_{\frac{i-j}{2}} \widehat{\Psi}_{\frac{i+j}{2}}}
     { \Psi_{\frac{i-j}{2}}^2 \Psi_{\frac{i+j}{2}}^2 \widehat{\Psi}_{i}  }
     = 2I \frac{1}{(x-I)(x+I)} = 3 \frac{x^2 - Ix - 2}{(x-I)} - 3 \frac{x^2 + Ix - 2}{(x+I)} = x' \circ \frac{i-j}{2}- x' \circ \frac{i+j}{2}.
\]

\section{Higher dimension}
\label{sec:higher}

The definitions and result for $\Psi_\phi$ can be extended to higher dimension, in the same fashion as for elliptic nets.
Let $\langle \phi, \phi' \rangle = \frac{1}{2} \left( \deg (\phi + \phi') - \deg \phi - \deg \phi' \right)$.
Define for any vector $\vec \phi = (\phi_1,\ldots, \phi_k)$ whose entries are isogenies $\phi_i: E \rightarrow E'$, an elliptic function $\Psi_{\vec \phi}$ on $E^k$ with divisor
\begin{equation}
	\label{eqn:phidiv}
	(\vec \phi \cdot \mathbf{P} = \mathbf{0})
	+ \sum_i (P_i = P_{\phi_i})
	- \sum_i (\deg \phi_i + 1) ({P}_i = \mathcal{O})
	- \sum_{i < j} \langle \phi_i, \phi_j \rangle \left( ({P}_i + {P}_j = \mathcal{O}) - (P_i = \mathcal{O}) - (P_j = \mathcal{O}) \right),
\end{equation}
and normalized in a manner similar, namely, where we denote by $\sigma$ the summation function $(P_1,\ldots,P_k) \mapsto P_1 + \cdots + P_k$, and by $\pi_i$ the projection onto the $i$-th component, and require
\begin{equation}
	\label{eqn:norm-arb-gen}
  \begin{split}
	&\frac{ 
	\Psi_{\vec \phi} 
\prod_i (t^{\deg \phi_i + \deg g_{\phi_i} - \sum_{j \neq i} \langle \phi_i, \phi_j \rangle} \circ \pi_i) 
	\prod_{i < j} t^{\langle \phi_i, \phi_j \rangle} \circ (\sigma \circ (\pi_i \times \pi_j)) 
	}{
(t' \circ \sigma \circ (\phi_1\times \cdots \times \phi_k))
	\prod_i (t_{\phi_i} \circ g_{\phi_i} \circ \pi_i) 
} (\mathcal{O})   \\
  &= \left( \frac{dt}{\omega} (\mathcal{O}) \right)^{\sum_i (\deg \phi_i + \deg g_{\phi_i}) - \sum_{i < j} \langle \phi_i, \phi_j \rangle}
	\left( \frac{dt'}{\omega'} (\mathcal{O}') \right)^{-1}
	\prod_i \left( \frac{dt_{\phi_i}}{\omega_{\phi_i}} (\mathcal{O}_{\phi_i}) \right)^{-1}.
\end{split}
\end{equation}

One can verify that, in each individual copy of $E$, the function above is an elliptic function.  
One can define
\[
	\widehat{\Psi}_{\vec \phi} = \prod_i \widehat{\Psi}_{\phi_i},
\]
and $\tilde{\Psi}_{\vec \phi}$ so $\widehat{\Psi}_{\vec \phi} \tilde{\Psi}_{\vec \phi} = \Psi_{\vec \phi}$.  
The formal group expansion becomes
\begin{equation*}
	\Psi_{\vec \phi}(T)
  = \left( \prod_i a_{\phi_i} a_{\iota(\phi_i)} \right) T^{-\sum_i (\deg \phi_i -\deg_{in} \phi_i+1) + \sum_{i < j} \langle \phi_i, \phi_j \rangle} + \cdots .
\end{equation*}
Lemma~\ref{lemma:quadrel} holds where the $\alpha$ are interpreted as vectors, and the relation to $x$ can be given as 
	\begin{equation*}
		\frac{
			\Psi_{\vec\alpha + \vec\beta} \Psi_{\vec\alpha - \vec\beta} \widehat{\Psi}_{\vec\alpha} \widehat{\Psi}_{\vec\beta}
		}{ 
		\Psi_{\vec\alpha}^2 \Psi_{\vec\beta}^2  \widehat{\Psi}_{\vec\alpha + \vec\beta} 
		} = x' \circ \sigma \circ \prod_i \beta_i - x' \circ \sigma \circ \prod_i \alpha_i.
	\end{equation*}
The first and second recurrence relations (Corollaries~\ref{cor:rec} and \ref{cor:rec2}) work out the same, where we interpret the indices as vectors of endomorphisms.

The final consideration is the chain rule, and one can show a version of the elliptic net chain rule for isogenies.  
Let $T: E''^k \rightarrow E^k$ be a linear transformation.  Then we have $T^{tr}: \Hom(E, E')^k \rightarrow \Hom(E'', E')^k$.

\begin{theorem}[First chain rule in higher dimension]
  Let $T$ be as above.
  Let $\vec \phi \in \Hom(E,E')^k$, i.e. $\phi_i: E \rightarrow E'$.
	Then whenever all coordinate isogenies in the subscripts are unbiased, we have
\[
	(\Psi_{\vec \phi} \circ T)
	\prod_{i=1}^k \Psi_{T^{tr}(\mathbf{e}_i)}^{\deg \phi_i - \sum_{j \neq i} \langle \phi_i, \phi_j \rangle}
	\prod_{i < j} \Psi_{T^{tr}(\mathbf{e}_i+\mathbf{e}_j)}^{\langle \phi_i, \phi_j \rangle}
	=
	\Psi_{T^{tr}(\vec \phi)}.
\]
In general, we have
\[
	\left(
	\frac{
	(\Psi_{\vec \phi} \circ T)
	\prod_{i=1}^k \Psi_{T^{tr}(\mathbf{e}_i)}^{\deg \phi_i - \sum_{j \neq i} \langle \phi_i, \phi_j \rangle}
	\prod_{i < j} \Psi_{T^{tr}(\mathbf{e}_i+\mathbf{e}_j)}^{\langle \phi_i, \phi_j \rangle}
	}{
	\Psi_{T^{tr}(\vec \phi)}
}\right)^2
	=
	\frac{
	(\widehat\Psi_{\vec \phi} \circ T)
	\prod_{i=1}^k \widehat\Psi_{T^{tr}(\mathbf{e}_i)}^{\deg \phi_i - \sum_{j \neq i} \langle \phi_i, \phi_j \rangle}
	\prod_{i < j} \widehat\Psi_{T^{tr}(\mathbf{e}_i+\mathbf{e}_j)}^{\langle \phi_i, \phi_j \rangle}
	}{
	\widehat\Psi_{T^{tr}(\vec \phi)}
}.
\]
\end{theorem}

\begin{theorem}[Second chain rule in higher dimension]
	\label{thm:seconchain-higher}
  Let $T$ be as above.
	Suppose 
 $\sum_{\vec\alpha \in \Hom(E'',E')^k} e_{\vec\alpha} q(\vec\alpha) = 0$ is a finite integral quadratic identity.
Suppose $e_{\vec \alpha} = 0$ whenever $\alpha$ is not in the image of $T^{tr}$.
Then
	\[
		\left(\prod_{\vec \gamma \in \Hom(E,E')^k} \Psi^{e_{T^{tr}(\vec \gamma)}}_{\vec \gamma}
		\sqrt{\prod_{\vec\gamma \in \Hom(E,E')^k} \widehat{\Psi}^{e_{T^{tr}(\vec\gamma)}}_{\vec\gamma}}\right) \circ T
		=
		\prod_{\vec\gamma \in \Hom(E,E')^k} \Psi^{e_{T^{tr}(\vec\gamma)}}_{T^{tr}(\vec\gamma)}
		\sqrt{\prod_{\vec\gamma \in \Hom(E,E')^k} \widehat{\Psi}^{e_{T^{tr}(\vec\gamma)}}_{T^{tr}(\vec\gamma)}}.
	\]
\end{theorem}

\section{Declarations.}

\subsection*{Data Availability Statement.}  We do not analyse or generate any datasets.
\subsection*{Funding.}   The author was supported by NSF DMS-2401580 and a Joan and Joseph Birman Fellowship from the American Mathematical Society.
\subsection*{Competing Interests.}  The author has no relevant financial or non-financial interests to disclose.
 
     \bibliographystyle{plain}
     \bibliography{ellnet}

     \end{document}